\documentclass[10pt]{ijnam}

\def\currentvolume{1}
\def\currentissue{1}
\def\currentyear{2004}
\def\month{March}
\pagespan{1}{18}
\copyrightinfo{2004}{} 

\usepackage{amsmath}
\usepackage{amsfonts}
\usepackage{amssymb}
\usepackage{graphicx}
\usepackage{color}
\usepackage{bm}
\usepackage{verbatim}

\allowdisplaybreaks

\newtheorem{theorem}{Theorem}[section]

\newtheorem{lemma}[theorem]{Lemma}
\newtheorem{assumption}[theorem]{Assumption}

\theoremstyle{definition}
\newtheorem{definition}[theorem]{Definition}

\begin{document}

\title[Ensemble-POD for the Navier-Stokes equations]{A higher-order ensemble/proper orthogonal\\ 
	 decomposition 
	method for the \\
	nonstationary Navier-Stokes equations}

\author[M. Gunzburger]{Max Gunzburger}
\address{
  Department of Scientific Computing,  
  Florida State University, Tallahassee, FL 32306-4120
}
\email{mgunzburger@fsu.edu}

\author[N. Jiang \and M. Schneier]{Nan Jiang \and Michael Schneier}
\address{
  Department of Scientific Computing,  
  Florida State University, Tallahassee, FL 32306-4120. Current address: Department of Mathematics and Statistics, Missouri University of Science and Technology, Rolla, MO 65409-0020
}
\email{jiangn@mst.edu}

\address{
	Department of Scientific Computing,  
	Florida State University, Tallahassee, FL 32306-4120
}
\email{mschneier89@gmail.com}

\date{January 1, 2004 and, in revised form, March 22, 2004.}


\abstract{Partial differential equations (PDE) often involve parameters, such as viscosity or
	density. An analysis of the PDE may involve considering a large range of parameter values,
	as occurs in uncertainty quantification, control and optimization, inference, and several statistical techniques. The solution for even
	a single case may be quite expensive; whereas parallel computing may be applied, this reduces the total elapsed
	time but not the total computational effort. In the case of flows governed by the Navier-Stokes equations,
	a method has been devised for computing an ensemble of solutions. Recently, a reduced-order model derived from a 
	proper orthogonal decomposition (POD) approach was incorporated into a first-order accurate in time version of the ensemble algorithm. In this work,
	we expand on that work by incorporating the POD reduced order model into a second-order accurate ensemble algorithm. 
	Stability and convergence results for this method are updated to account for the POD/ROM approach. 
	Numerical experiments illustrate the accuracy and efficiency of the new approach.}


\maketitle

\section{Introduction}
In science and engineering, mathematical models are utilized to understand and predict the behavior of complex systems. Two common input data/parameters for these types of models include forcing terms and initial conditions. Often, for any number of reasons, there is a degree of uncertainty involved with the specification of such inputs. In order to obtain an accurate model we must incorporate such uncertainties into the governing equations and quantify their effects on the outputs of the simulation. 

In this uncertainty quantification setting, one needs to determine many realizations of the outputs of the simulation in order to determine accurate statistical information about those outputs. A similar need occurs in other settings such as inference, optimization, and control and the need to determine solution ensembles arising in many applications, e.g., weather forecasting and turbulence modeling. Thus, in this work, we are interested in computing ensembles of solutions for the Navier-Stokes equations (NSE) with uncertainty present in the initial conditions and body forces. Specifically, for $j=1,\ldots,J$, we have
\begin{equation}\label{eq:NSE}
\left\{\begin{aligned}
u_{t}^j+u^{j}\cdot\nabla u^{j}-\nu\triangle u^{j}+\nabla p^{j}  &
=f^{j}(x,t)&\quad\forall x\in\Omega\times(0,T]\\
\nabla\cdot u^{j}  &  =0&\quad\forall x\in\Omega\times(0,T]\\
u^{j}  &  =0&\quad\forall x\in\partial\Omega\times(0,T]\\
u^{j}(x,0)  &  =u^{j,0}(x)&\quad\forall x\in\Omega,
\end{aligned}\right.
\end{equation}
where $\Omega \subset \mathbb{R}^{d}$, $d= 2,3$, is an open regular domain.
In most settings, in order to {\color{red} guarantee a desired accuracy level} in the outputs, a fine spatial resolution is usually required, which renders each realization to be computationally intensive. Traditionally, the simulation for each ensemble member is treated as a separate problem; therefore the focus has been to cut down on the total number of realizations needed. However, in recent works \cite{J15, J17, JL14,AMJ16}, new algorithms were designed that allow for the realizations to be computed simultaneously at each time step. In those papers, the focus has been on creating algorithms which allow for the same linear system to be used for all right-hand sides. Hence, the need to solve a different linear system for each right-hand side is reduced to solving the a single system with many different right-hand sides. This is a well studied problem for which efficient block iterative methods already exist. A few examples of these include block CG \cite{FOP95}, block QMR \cite{FM97}, and block GMRES \cite{GS96}.

The next natural step to take to further improve upon the efficiency of these ensemble algorithms is the introduction of reduced-order modeling (ROM) techniques. Specifically, we are interested in the implementation of the proper orthogonal decomposition (POD) approach into the ensemble framework. POD works by generally extracting from highly accurate numerical simulations or experimental data the most energetic modes in a given system. 

One can then project the original problem onto these POD modes to obtain the Galerkin proper orthogonal decomposition (POD-G-ROM) approximation to the original problem. It has been shown for laminar flows \cite{BGH,IW14} that it is possible to obtain a good approximation with few POD modes, hence the corresponding POD-G-ROM only requires the solution of a small linear system. Recently, in \cite{GJS17}, a POD-G-ROM ensemble algorithm based on the first-order accurate in time ensemble algorithm first introduced in \cite{JL14} was developed. However, in applications that require long-term time integration, such as climate and ocean forecasting, higher-order methods are highly desirable. For this reason, in this paper we expand on the work of \cite{GJS17} by introducing the POD method into the second-order ensemble algorithm introduced in \cite{J15}.

\section{Notation and preliminaries}

We denote by $\|\cdot\|$ and $(\cdot,\cdot)$ the $L^{2}(\Omega)$ norm and inner product, respectively, and by $\|\cdot\|_{L^{p}}$ and $\|\cdot\|_{W_{p}^{k}}$ the $L^{p}(\Omega)$ and Sobolev
$W^{k}_{p}(\Omega)$ norms, 
respectively. $H^{k}(\Omega)=W_{2}^{k}(\Omega)$ with
norm $\|\cdot\|_{k}$. For a function $v(x,t)$ that is well defined on $\Omega \times [0,T]$, 
we define the norms
$$
|||v|||_{2,s} : = \Big(\int_{0}^{T}\|v(\cdot,t)\|_{s}^{2}dt\Big)^{\frac{1}{2}}
\qquad \text{and} \qquad 
\||v|||_{\infty,s} := \text{ess\,sup}_{[0,T]}\|v(\cdot,t)\|_{s} .
$$
The space $H^{-1}(\Omega)$ denotes the dual space of bounded linear functionals defined on $H^{1}_{0}(\Omega)=\{v\in H^{1}(\Omega)\,:\,v=0 \mbox{ on } \partial\Omega\}$; this space is equipped with the norm
$$
\|f\|_{-1}=\sup_{0\neq v\in X}\frac{(f,v)}{\| \nabla v\| } 
\quad\forall f\in H^{-1}(\Omega).
$$

The solutions spaces $X$ for the velocity and $Q$ for the pressure are respectively defined as
$$
\begin{aligned}
X : =& [H^{1}_{0}(\Omega)]^{d} = \{ v \in [L^{2}(\Omega)]^{d} \,:\, \nabla v \in [L^{2}(\Omega)]^{d \times d} \ \text{and} \  v = 0 \ \text{on} \ \partial \Omega \} \\
Q : =& L^{2}_{0}(\Omega) = \Big\{ q \in L^{2}(\Omega) \,:\, \int_{\Omega} q dx = 0 \Big\}.
\end{aligned}
$$
A weak formulation of (\ref{eq:NSE}) is given as follows: for $j=1, \ldots, J$, find $u^j:(0,T]\rightarrow X$ and $p^j:(0,T]\rightarrow Q$ such that, for almost all $t\in(0,T]$, satisfy 
\begin{equation}\label{wfwf}
\left\{\begin{aligned}
(u_{t}^j,v)+(u^{j}\cdot\nabla u^{j},v)+\nu(\nabla u^{j},\nabla v)-(p^{j}
,\nabla\cdot v)  &  =(f^{j},v)&\quad\forall v\in X\\
(\nabla\cdot u^{j},q)  &  =0&\quad\forall q\in Q\\
u^{j}(x,0)&=u^{j,0}(x).&
\end{aligned}\right.
\end{equation}
The subspace of $X$ consisting of weakly divergence-free functions is defined as
$$
V :=\{v\in X \,:\,(\nabla\cdot v,q)=0\,\,\forall q\in Q\} \subset X.
$$

We denote conforming velocity and pressure finite element spaces based on a regular triangulation of $\Omega$ having maximum triangle diameter $h$ by
$
X_{h}\subset X$ {and} $ Q_{h}\subset Q.
$
We assume that the pair of spaces $(X_h,Q_h)$ satisfy the discrete inf-sup (or $LBB_h$) condition required for stability of finite element approximations; we also assume that the finite element spaces satisfy the approximation properties
$$
\begin{aligned}
\inf_{v_h\in X_h}\| v- v_h \|&\leq C h^{s+1}&\forall v\in [H^{s+1}(\Omega)]^d\\
\inf_{v_h\in X_h}\| \nabla ( v- v_h )\|&\leq C h^s&\forall v\in [H^{s+1}(\Omega)]^d\\
\inf_{q_h\in Q_h}\|  q- q_h \|&\leq C h^s&\forall q\in H^{s}(\Omega),
\end{aligned}
$$
where $C$ is a positive constant that is independent of $h$. The Taylor-Hood element pairs ($P^s$-$P^{s-1}$), $s\geq 2$, are one common choice for which the $LBB_h$ stability condition and the approximation estimates hold \cite{GR79, Max89}.

In this paper, we solve the NSE \eqref{eq:NSE} (to generate snapshots) using a second-order time stepping scheme (e.g., Crank-Nicolson). Thus, we assume the finite element approximations satisfy the error estimates, which are used in the analysis in Section 5,
\begin{gather*}
\| u- u_h \|\leq C (h^{s+1}+\Delta t^{2} )\\
\| \nabla ( u- u_h ) \|\leq C (h^s+\Delta t^{2} ).
\end{gather*}

We define the trilinear form
$$
b(w,u,v) = (w\cdot\nabla u,v) 
\qquad\forall u,v,w\in [H^1(\Omega)]^d
$$
and the explicitly skew-symmetric trilinear form given by 
$$
b^{\ast}(w,u,v):=\frac{1}{2}(w\cdot\nabla u,v)-\frac{1}{2}(w\cdot\nabla v,u)
\qquad\forall u,v,w\in [H^1(\Omega)]^d
$$
which satisfies the bounds \cite{Layton08}
\begin{gather}
b^{\ast}(w,u,v)\leq C \|  \nabla w\|   \| \nabla u\| (\|  v \|  \| \nabla
v \| )^{1/2}\qquad\forall u, v, w \in X \label{In1}\\
b^{\ast}(w,u,v)\leq C (\| w \| \|  \nabla w\| )^{1/2}  \| \nabla u\|  \| \nabla
v \| \qquad\forall u, v, w \in X .\label{In2}
\end{gather}
We also have for $\forall u_h, v_h, w_h\in X_h$
\begin{align}\label{b-equality}
b^{\ast}(w_h,u_h, v_h)=\int_{\Omega}w_h \cdot \nabla u_h \cdot v_h dx+ \frac{1}{2}\int_{\Omega}(\nabla \cdot w_h)(u_h\cdot v_h) dx.
\end{align}
We also define the discretely divergence-free space $V_h$ as
$$
V_{h} :=\{v_{h}\in X_{h}\,:\,(\nabla\cdot v_{h},q_{h})=0\,\,\forall
q_{h}\in Q_{h}\}  \subset X.
$$
In most cases, and for the Taylor-Hood element pair in particular, $V_{h} \not\subset V$, i.e., discretely divergence-free functions are not divergence-free.

In this work, we focus on extending our  previously developed second-order accurate approximation for computing the solution of an ensemble of Navier-Stokes equations. These methods require a new definition for the ensemble mean as opposed to that used in \cite{GJS17}.

\begin{definition}\label{def21}
	Let $t^{n}=n\Delta t$, $n=0,1,2,\ldots,N$, where $N:=T/\Delta t$, denote a partition of the interval $[0,T]$. Let $|\cdot|$ denote the Euclidian length of a vector. For $j=1, \ldots, J$ and $n=0,1,2,\ldots,N$, let $u^{j,n}(x):=u^{j}(x,t^{n})$. Then, the \text{\bf ensemble mean}, \text{\bf fluctuation about the mean}, and \text{\bf magnitude of fluctuation}  are respectively defined, for $n=1,2,\ldots,N$, by 
	\begin{equation*}
	\left\{\begin{aligned}
	<u>^n : &=\frac{1}{J}\sum_{j=1}^{J}\left(2u^{j,n} - u^{j,n-1}\right)\\
	u'^{j,n} :&= 2u^{j,n} - u^{j,n-1} - <u>^{n}\\
	|u'^{n}| :&= \Big(\sum_{j=1}^{J}|u'^{j,n}|^{2}\Big)^{\frac{1}{2}}.
	\end{aligned}\right.
	\end{equation*}
\end{definition}

The full space-time discrete model given in \cite{JL15} that we refer to as \emph{EnB-full} is given as: for $j = 1, \ldots,J$, given $u^{j,0}_h \in X_h$ and $u^{j,1}_h\in X_h$, for $n=1,2,\ldots,N-1$ find $u^{j,n+1}_h\in X_h$ and $p_h^{j,n+1}\in Q_h$ satisfying
\begin{equation}\label{EnB}\emph{}
\begin{aligned}
&\Big(\frac{3u^{j,n+1}_h - 4u^{j,m}_h + u^{j,n-1}_h}{2\Delta t}, v_h \Big) + b^{\ast}(<u_h>^{n} , u^{j,n+1}_h ,v_h)+ \nu (\nabla u^{j,n+1}_h, \nabla v_h)\\
& + b^{\ast}(u'^{j,n}_h ,2u^{j,n}_h - u^{j,n-1}_h, v_h) - (p^{j,n+1}_h , \nabla \cdot v_h)  =( f^{j,n+1}, v_h) \quad \forall v_h\in X_h\\
&\qquad\quad (\nabla \cdot u_h^{j,n+1}, q_h )= 0 \qquad \forall q_h\in Q_h.
\end{aligned}
\end{equation}

\section{Proper orthogonal decomposition (POD) method}
\label{POD_sec}

We next briefly discuss the POD method and how we can apply it to the previously discussed ensemble algorithm. For a more in depth description, see \cite{GPS07, BGH, KV01, KV02}. In Section \ref{numex}, the results of numerical experiments are given to illustrate the efficacy of our new algorithm.

Given a positive integer $N_S$, let $0=t_0<t_1< \cdots < t_{N_S} = T$ denote a uniform partition of the time interval $[0,T]$. For $j=1,\ldots,J_S$, we select $J_S$ different initial conditions $u^{j,0}(x)$ and denote by $u_{h,S}^{j,m}(x)\in X_h$, $j=1,\ldots,J_S$, $m=1,\ldots,N_S$, the finite element approximation to \eqref{wfwf} evaluated at $t=t_m$, $m=1,\ldots,N_S$, corresponding to an approximate initial condition $u_h^{j,0}(x)\in X_h$ such as the $X_h$-interpolant of $u^{j,0}(x)$. We then define the space spanned by the $J_S(N_S+1)$ discrete snapshots as
\begin{equation*}
X_{h,S}:=\text{span} \{  u_{h,S}^{j,m}(x) \}_{j=1,m=0}^{J_S,N_S} \subset V_h \subset X_h.
\end{equation*}

Denoting by $\vec{u}_S^{j,m}$ the vector of coefficients corresponding to the finite element function $u_{h,S}^{j,m}(x)$, where $K=\dim X_h$, we define the $K\times J_S(N_S+1)$ {\em snapshot matrix} $\mathbb{A}$ as
$$
\mathbb{A} = \big(\vec{u}_S^{1,0},\vec{u}_S^{1,1}, \ldots , \vec{u}_S^{1,N_S}, \vec{u}_S^{2,0},\vec{u}_S^{2,1},  \ldots , \vec{u}_S^{2,N_S}, \ldots , \vec{u}_S^{J_S,0},\vec{u}_S^{J_S,1}, \ldots , \vec{u}_S^{J_S,N_S}\big),
$$
i.e., the columns of $\mathbb{A}$ are the finite element coefficient vectors corresponding to the discrete snapshots. The POD method then seeks a low dimensional basis 
$$
X_R :=\text{span}\{{\varphi}_i\}_{i=1}^R \subset X_{h,S} \subset V_h\subset X_h
$$
which can approximate the snapshot data. This basis can be determined by finding an orthonormal basis $\{\varphi_i\}_{i=1}^S$ for $X_{h,S}$ such that for all $R\in \{1,\ldots,S\}$, $\{\varphi_i\}_{i=1}^R$, solves the constrained minimization problem
\begin{equation}\label{Min}
\begin{aligned}
\min  \sum_{k=1}^{J_S} \sum_{l=0}^{N_S}\Big \|  u_{h,s}^{k,l}-\sum_{j=1}^R (u_{h,s}^{k,l}, \varphi_j)\varphi_j\Big \| ^2 \\
\text{subject to } (\varphi_i, \varphi_j)= \delta_{ij}\quad\mbox{for $i,j=1,\ldots,R$},
\end{aligned}
\end{equation}
where $\delta_{ij}$ denotes the Kronecker delta and the minimization is with respect to all orthonormal bases for $X_{h,S}$. We define the $J_S(N_S+1)\times J_S(N_S+1)$ correlation matrix $\mathbb{C} = \mathbb{A}^{T}\mathbb{M}\mathbb{A}$, where $\mathbb{M}$ denotes the Gram matrix corresponding to full finite element space. Then, the problem \eqref{Min} is equivalent to determine the $R$ dominant eigenpairs $\{\lambda_i,\vec{a}_i\}$ satifying
\begin{equation*}
\label{eigProb}
\mathbb{C}\vec{a}_{i} = \lambda_{i}\vec{a}_{i}, \quad |\vec{a}_{i}| = 1, \quad \vec{a}^{T}_{i}\vec{a}_{j} = 0 \ \text{if} \ i \neq j, \quad \text{and} \quad \lambda_{i} \geq \lambda_{i-1} > 0,
\end{equation*}
where $|\cdot|$ denotes the Euclidean norm of a vector. The finite element coefficient vectors corresponding to the POD basis functions are then given by
\begin{equation*}
\vec\varphi_i = \frac{1}{\sqrt{\lambda_i}}\mathbb{A}\vec{a}_{i}, \ \ \ i = 1, \ldots, R.
\end{equation*}

\subsection{POD-ensemble based reduced-order model}
\label{pod_gen}

We next show how the POD method can be used to construct a reduced-order model within the framework of the {EnB-full} algorithm. The formulation of the POD  approximation is identical to that of the full finite element approximation. Instead of seeking a solution in the finite element space $X_{h}$, we seek a solution in the POD space $X_{R}$ using the basis $\{ \varphi_{i}\}_{i=1}^{R}.$ Now, for $j=1,\ldots,J$, we define the POD approximation of the initial conditions as $u_{R}^{j,0}(x)=\sum_{i=1}^R (u^{j,0}, {\varphi}_i)\varphi_i(x)\in X_R$ and POD approximation of the of $u^{j,1}$ as $u_{R}^{j,1}(x)=\sum_{i=1}^R (u^{j,1}_h, {\varphi}_i)\varphi_i(x)\in X_R$, where $u_h^{j,1}$ denotes a finite element approximation of $u^{j,1}$ using a second-order (or higher-order) time stepping method such that $\Vert u^{j,1}-u_h^{j,1}\Vert_1=O(h^s + \Delta t^2)$ accurate or better. We then pose the following problem: {\em given $u_{R}^{j,0}(x), u_R^{j,1}\in X_R$, for $n=0,1,\ldots,N-1$ and for $j=1,\ldots,J$, find $u_{R}^{j,n+1}\in X_R$ satisfying}
\begin{equation}\label{EnB-POD-Weak}
\begin{aligned}
&\big(\frac{3u_{R}^{j,n+1}-4u_{R}^{j,n} + u_{R}^{j,n-1}}{2\Delta t}, \varphi\big)+b^{\ast}(<u_{R}
>^{n},u_{R}^{j,n+1},\varphi)\\
&\quad
+b^{\ast}(u_{R}'^{j, n},2u_{R}^{j,n} - u_{R}^{j,n-1},\varphi)+\nu(\nabla u_{R}^{j,n+1},\nabla
\varphi)=(f^{j,n+1},\varphi)\qquad\forall \varphi\in X_{R}.
\end{aligned}
\end{equation}
%
We refer to  \eqref{EnB-POD-Weak} as \emph{EnB-POD}. We note that because $X_{R} \subset V_{h}$, the POD basis is discretely divergence free by construction. Therefore, there is no pressure term present in \eqref{EnB-POD-Weak} and the only unknown in our system is the approximate velocity vector. 

\section{Stability of EnB-POD}

In this section, we present results pertaining to the conditional, nonlinear, longterm stability of solutions of \eqref{EnB-POD-Weak}.

We denote by $\|\hspace{-1pt}|  \cdot \|\hspace{-1pt}|_2 $ the spectral norm for symmetric matrices and let ${\mathbb M}_R$ denote the $R\times R$ POD mass matrix with entries $[{\mathbb M}_R]_{i,i'}= (\varphi_i, \varphi_{i'})$ and ${\mathbb S}_R$ denote the $R\times R$ matrix with entries $[{\mathbb S}_R]_{i,i'}=[{\mathbb M}_R]_{i,i'}+\nu(\nabla \varphi_i, \nabla \varphi_{i'})$, $i,i'=1,\ldots,R$. By construction the POD basis functions are orthonormal with respect to the $L^{2}(\Omega)$ inner product, hence the mass matrix is actually the identity matrix and $\|\hspace{-1pt}|  {\mathbb M}_R \|\hspace{-1pt}| _2 = \|\hspace{-1pt}|  {\mathbb M}_R^{-1} \|\hspace{-1pt}| _2 = 1$. Then, we have the inequality \cite{KV01} 
\begin{equation}\label{lm:inverse}
\|  \nabla \varphi \|  \leq \big(\|\hspace{-1pt}|  {\mathbb S}_R \|\hspace{-1pt}| _2 \|\hspace{-1pt}|  {\mathbb M}_R^{-1} \|\hspace{-1pt}| _2\big)^{1/2} \| \varphi \|  \leq \|\hspace{-1pt}|  {\mathbb S}_R \|\hspace{-1pt}| _2 ^{1/2} \| \varphi \|\qquad \forall \varphi \in X_R.
\end{equation}

We now show that \emph{EnB-POD} is stable under a time-step condition similar to that for \emph{EnB-full} except the norm of the POD stiffness matrix now plays a role in the stability condition.

\begin{theorem}
	\label{global_stability}
	For the method \eqref{EnB-POD-Weak}, suppose that 
	\begin{equation}
	\label{StabCondEnB}
	C\|\hspace{-1pt}|  {\mathbb S}_R \|\hspace{-1pt}| _2^{\frac{1}{2}}\frac{\Delta t}{\nu}\| \nabla u'^{j,n}_{R}\|^{2} \leq 1 \qquad \mbox{for $j=1,\ldots,J$}.
	\end{equation}
	Then, for $ n = 0, \ldots, N-1$
	\begin{equation}\label{Stability}
	\begin{aligned}
	&\frac{1}{4}\|u^{j,n}_{R}\|^{2} + \frac{1}{4}\|2u^{j,n}_{R} - u^{j,n-1}_{R}\|^{2} +\frac{\Delta t \nu}{4} \sum_{n'=1}^{n-1} \| \nabla u_{R}^{j,n'+1} \|^{2} \\
	&\leq \frac{\Delta t}{\nu}\sum_{n'=1}^{n-1}\|f^{j,n'+1}\|^{2}_{-1} + \frac{1}{4}\|u^{j,1}_{R}\|^{2} + \frac{1}{4}\|2u^{j,1}_{R} - u^{j,0}_{R}\|^{2}.
	\end{aligned}
	\end{equation}
\end{theorem}

\begin{proof} 
	The proof is very similar to that for \cite[Theorem 1]{J15}, the main difference being the use of a different basis function. Setting $\varphi = u^{j,n+1}_{R}$, using skew-symmetry of the trilinear form, and applying Young's inequality to the right-hand side, we have.
	\begin{equation}
	\label{Stab1}
	\begin{aligned}
	&\frac{1}{4}\left(\|u^{j,n+1}_{R}\|^{2} + \|2u^{j,n+1}_{R} - u^{j,n}_{R}\|^{2}\right) - \frac{1}{4}\left(\|u^{j,n}_{R}\|^{2} + \|2u^{j,n}_{R} - u^{j,n-1}_{R}\|^{2}\right)
	\\ &+\frac{1}{4}\|u^{j,n+1}_{R} - 2u^{j,n}_{R} + u^{j,n-1}_{R}\|^{2} + \Delta t b^{\ast}(u'^{j,n}_{R},2u^{j,n}_{R} - u^{j,n-1}_{R}, u^{j,n+1}_{R}) 
	\\ &+\nu \Delta t \| \nabla u_{R}^{j,n+1} \|^{2} \leq \frac{\nu \Delta t}{4}\|\nabla u^{j,n+1}_{R}\|^{2} + \frac{\Delta t}{\nu}\|f^{j,n+1}\|^{2}_{-1}.
	\end{aligned}
	\end{equation}
	We can then rewrite, using skew symmetry, the trilinear term  as
	\begin{equation*}
	\begin{aligned}
	b^{\ast}(u_{R}'^{j, n},2u_{R}^{j,n} - u_{R}^{j,n-1},u_{R}^{j,n+1}) = b^{\ast}(u'^{j,n}_{R},u^{j,n+1}_{R}, u^{j,n+1}_{R} - 2u^{j,n}_{R} + u^{j,n-1}_{R}).
	\end{aligned}
	\end{equation*}
	Now bounding the trilinear form using \eqref{lm:inverse}
	\begin{equation}
	\label{Stab2}
	\begin{aligned}
	&\Delta t b^{\ast}(u'^{j,n}_{R},u^{j,n+1}_{R}, u^{j,n+1}_{R} - 2u^{j,n}_{R} + u^{j,n-1}_{R}) \\
	&\leq C \Delta t \|\nabla u'^{j,n}_{R}\||| \nabla u^{j,n+1}_{R}\| \| u^{j,n+1}_{R} - 2u^{j,n}_{R} + u^{j,n-1}_{R} \|^{\frac{1}{2}} \| \nabla (u^{j,n+1}_{R} - 2u^{j,n}_{R} + u^{j,n-1}_{R}) \|^{\frac{1}{2}}\\
	&\leq C\Delta t \|\nabla u'^{j,n}_{R}\|||\nabla u^{j,n+1}_{R}\| \| u^{j,n+1}_{R} - 2u^{j,n}_{R} + u^{j,n-1}_{R} \| \|\hspace{-1pt}|  {\mathbb S}_R \|\hspace{-1pt}| _2^{\frac{1}{4}}.
	\end{aligned} 
	\end{equation} 
	
	Applying Young's inequality gives
	\begin{equation*}
	\begin{aligned}
	C&\Delta t \|\nabla u'^{j,n}_{R}\||| \nabla u^{j,n+1}_{R}\| \| u^{j,n+1}_{R} - 2u^{j,n}_{R} + u^{j,n-1}_{R} \|||\hspace{-1pt}|  {\mathbb S}_R \|\hspace{-1pt}| _2^{\frac{1}{4}} \\
	& \leq C\Delta t^{2} \|\nabla u'^{j,n}_{R}\|^{2}\| \nabla u^{j,n+1}_{R}\|^{2}\|\hspace{-1pt}|  {\mathbb S}_R \|\hspace{-1pt}| _2^{\frac{1}{2}} + \frac{1}{4}\| u^{j,n+1}_{R} - 2u^{j,n}_{R} + u^{j,n-1}_{R} \|^{2}. 
	\end{aligned}
	\end{equation*}
	Using this bound we then combine like terms. \eqref{Stab1} becomes
	\begin{equation}
	\label{Stab3}
	\begin{aligned}
	&\frac{1}{4}\left(\|u^{j,n+1}_{R}\|^{2} + \|2u^{j,n+1}_{R} - u^{j,n}_{R}\|^{2}\right) - \frac{1}{4}\left(\|u^{j,n}_{R}\|^{2} + \|2u^{j,n}_{R} - u^{j,n-1}_{R}\|^{2}\right) \\ 
	&+ \frac{\nu \Delta t}{4} \| \nabla u_{R}^{j,n+1} \|^{2} + \frac{\nu \Delta t}{2}\left(1 - C\|\hspace{-1pt}|  {\mathbb S}_R \|\hspace{-1pt}| _2^{\frac{1}{2}}\frac{\Delta t}{\nu}\| \nabla u'^{j,n}_{R}\|^{2}\right)\|\nabla u^{j,n+1}_{R}\|\\ &\qquad\leq \frac{\Delta t}{\nu}\|f^{j,n+1}\|^{2}_{-1}.
	\end{aligned}
	\end{equation}
	Now under condition \eqref{StabCondEnB}, it follows that
	\begin{equation*}
	\frac{\nu \Delta t}{2}\left(1 - C\|\hspace{-1pt}|  {\mathbb S}_R \|\hspace{-1pt}| _2^{\frac{1}{2}}\frac{\Delta t}{\nu}\| \nabla u'^{j,n}_{R}\|^{2}\right)\|\nabla u^{j,n+1}_{R}\| \geq 0 
	\end{equation*} 
	\eqref{Stab3} and then simplifies to
	\begin{equation*}
	\begin{aligned}
	&\frac{1}{4}\left(\|u^{j,n+1}_{R}\|^{2} + \|2u^{j,n+1}_{R} - u^{j,n}_{R}\|^{2}\right) - \frac{1}{4}\left(\|u^{j,n}_{R}\|^{2} + \|2u^{j,n}_{R} - u^{j,n-1}_{R}\|^{2}\right)\\ 
	&\qquad+ \frac{\nu \Delta t}{4} \| \nabla u_{R}^{j,n+1} \|^{2} \leq \frac{\Delta t}{\nu}\|f^{j,n+1}\|^{2}_{-1}.
	\end{aligned}
	\end{equation*}
	Summing up the above inequality results in \eqref{Stability}.
\end{proof}

\subsection{Improving the time-step condition}

While Theorem \ref{global_stability} is true universally in two and three dimensions, there are also certain cases for which we can improve upon our estimate. We are specifically interested in the three-dimensional case for which we make use of the Sobolev embedding inequality: for $v \in X$ and $d = 3$, the $L^{6}$ norm satisfies \cite{Ladyzhenskaya69}
\begin{equation}
\label{sobolev_ineq}
\| v \|_{L^{6}} \leq C_{se} \| \nabla v \|  \quad \forall v \in X,
\end{equation}
where the constant $C_{se} = \frac{2}{\sqrt3}$. In the following stability estimate, we  see that instead of needing to control the gradient of fluctuations about the mean, we now are only concerned with $L^{3}$ fluctuations about the mean.

\begin{theorem}
	\label{3d_stability}
	Consider the method \eqref{EnB-POD-Weak} when $d = 3$. Suppose that 
	\begin{equation}
	\label{StabCond3D}
	C_{se}^{2}\|\mathbb{S}_{R}\|\frac{\Delta t}{\nu}\| u'^{j,n}_{R}\|_{L^{3}}^{2} \leq 1 \qquad \mbox{for $j=1,\ldots,J$}.
	\end{equation}
	Then, for $ n = 0, \ldots, N-1$
	\begin{equation}\label{Stability3D}
	\begin{aligned}
	&\frac{1}{4}\|u^{j,n}_{R}\|^{2} + \frac{1}{4}\|2u^{j,n}_{R} - u^{j,n-1}_{R}\|^{2} +\frac{\Delta t \nu}{4} \sum_{n'=1}^{n-1} \| \nabla u_{R}^{j,n'+1} \|^{2} \\
	&\leq \frac{\Delta t}{\nu}\sum_{n'=1}^{n-1}\|f^{j,n'+1}\|^{2}_{-1} + \frac{1}{4}\|u^{j,1}_{R}\|^{2} + \frac{1}{4}\|2u^{j,1}_{R} - u^{j,0}_{R}\|^{2}.
	\end{aligned}
	\end{equation}
\end{theorem}

\begin{proof}
	By \eqref{b-equality} and H$\ddot{o}$lders' inequality, we have
	\begin{equation*}
	\begin{aligned}
	b^{\ast}(u'^{j,n}_{R},&u^{j,n+1}_{R}, u^{j,n+1}_{R} - 2u^{j,n}_{R} + u^{j,n-1}_{R}) \\
	&\leq\frac{1}{2}\|u'^{j,n}_{R}\|_{L^{3}}\|\nabla u^{j,n + 1}_{R}\|_{L^{2}}\|u^{j,n+1}_{R} - 2u^{j,n}_{R} + u^{j,n-1}_{R}\|_{L^{6}} \\
	&\quad+\frac{1}{2}\|u'^{j,n}_{R}\|_{L^{3}}\|u^{j,n + 1}_{R}\|_{L^{6}}\|\nabla(u^{j,n+1}_{R} - 2u^{j,n}_{R} + u^{j,n-1}_{R})\|_{L^{2}}.
	\end{aligned}
	\end{equation*}
	Now applying \eqref{sobolev_ineq} and \eqref{lm:inverse} we have 
	
	\begin{align*}
	&\|u^{j,n+1}_{R} - 2u^{j,n}_{R} + u^{j,n-1}_{R}\|_{L^{6}} \leq C_{se}\|\hspace{-1pt}|  {\mathbb M}_R^{-1} \|\hspace{-1pt}| _2^{\frac{1}{2}}\|\hspace{-1pt}|  {\mathbb S}_R \|\hspace{-1pt}| _2^{\frac{1}{2}} \|u^{j,n+1}_{R} - 2u^{j,n}_{R} + u^{j,n-1}_{R}\| \\
	&\|u^{j,n+1}_{R}\|_{L^{6}} \leq C_{se} \|\nabla u^{j,n+1}_{R}\| \\
	&\|\nabla(u^{j,n+1}_{R} - 2u^{j,n}_{R} + u^{j,n-1}_{R})\|_{L^{2}} \leq \|\hspace{-1pt}|  {\mathbb M}_R^{-1} \|\hspace{-1pt}| _2^{\frac{1}{2}}\|\hspace{-1pt}|  {\mathbb S}_R \|\hspace{-1pt}| _2^{\frac{1}{2}} \|u^{j,n+1}_{R} - 2u^{j,n}_{R} + u^{j,n-1}_{R}\| 
	\end{align*}
	By construction the POD basis functions are orthonormal with respect to the $L^{2}(\Omega)$ inner product, hence $\|\hspace{-1pt}|  {\mathbb M}_R \|\hspace{-1pt}| _2 = \|\hspace{-1pt}|  {\mathbb M}_R^{-1} \|\hspace{-1pt}| _2 = 1$. It then follows that 
	\begin{equation*}
	\begin{aligned}
	\Delta t
	b^{\ast}(u'^{j,n}_{R},&u^{j,n+1}_{R}, u^{j,n+1}_{R} - 2u^{j,n}_{R} + u^{j,n-1}_{R})  \\
	&\leq C_{se} \|\hspace{-1pt}|  {\mathbb S}_R \|\hspace{-1pt}| _2 ^{\frac{1}{2}} \Delta t \|u'^{j,n}_{R}\|_{L^{3}} \|\nabla u^{j,n+1}_{R}\| \|u^{j,n+1}_{R} - 2u^{j,n}_{R} + u^{j,n-1}_{R}\|.
	\end{aligned}
	\end{equation*}
	Now applying Young's inequality, we have
	\begin{equation*}
	\begin{aligned}
	C_{se} & \|\hspace{-1pt}|  {\mathbb S}_R \|\hspace{-1pt}| _2^{\frac{1}{2}} \Delta t \|u'^{j,n}_{R}\|_{L^{3}} \|\nabla u^{j,n+1}_{R}\| \|u^{j,n+1}_{R} - 2u^{j,n}_{R} + u^{j,n-1}_{R}\|  \\ 
	&\leq C_{se}^{2} \|\hspace{-1pt}|  {\mathbb S}_R \|\hspace{-1pt}| _2 \Delta t^{2} \|u'^{j,n}_{R}\|^{2}_{L^{3}} \|\nabla u^{j,n+1}_{R}\|^{2} + \frac{1}{4} \|u^{j,n+1}_{R} - 2u^{j,n}_{R} + u^{j,n-1}_{R}\|^2.
	\end{aligned}
	\end{equation*}
	We then combine like terms and have 
	\begin{equation}
	\label{Stab5}
	\begin{aligned}
	&\frac{1}{4}\left(\|u^{j,n+1}_{R}\|^{2} + \|2u^{j,n+1}_{R} - u^{j,n}_{R}\|^{2}\right) - \frac{1}{4}\left(\|u^{j,n}_{R}\|^{2} + \|2u^{j,n}_{R} - u^{j,n-1}_{R}\|^{2}\right) \\
	&+ \frac{\nu \Delta t}{4} \| \nabla u_{R}^{j,n+1} \|^{2} + \frac{\nu \Delta t}{2}\left(1 - C_{se}^{2}\|\hspace{-1pt}|  {\mathbb S}_R \|\hspace{-1pt}| _2\frac{\Delta t}{\nu}\| u'^{j,n}_{R}\|_{L^{3}}^{2}\right)\|\nabla u^{j,n+1}_{R}\| \\
	&\leq \frac{\Delta t}{\nu}\|f^{j,n+1}\|^{2}_{-1}.
	\end{aligned}
	\end{equation}
	Now under condition \eqref{StabCond3D}, it follows that
	\begin{equation*}
	\frac{\nu \Delta t}{2}\left(1 - C_{se}^{2}\|\hspace{-1pt}|  {\mathbb S}_R \|\hspace{-1pt}| _2\frac{\Delta t}{\nu}\| u'^{j,n}_{R}\|_{L^{3}}^{2}\right)\|\nabla u^{j,n+1}_{R}\| \geq 0. 
	\end{equation*} 
	\eqref{Stab5} then simplifies to
	\begin{equation*}
	\begin{aligned}
	&\frac{1}{4}\left(\|u^{j,n+1}_{R}\|^{2} + \|2u^{j,n+1}_{R} - u^{j,n}_{R}\|^{2}\right) - \frac{1}{4}\left(\|u^{j,n}_{R}\|^{2} + \|2u^{j,n}_{R} - u^{j,n-1}_{R}\|^{2}\right) \\
	&\qquad+ \frac{\nu \Delta t}{4} \| \nabla u_{R}^{j,n+1} \|^{2} \leq \frac{\Delta t}{\nu}\|f^{j,n+1}\|^{2}_{-1}.
	\end{aligned}
	\end{equation*}
	Summing up the above inequality results in \eqref{Stability3D}.
\end{proof}

\section{Error analysis of En-POD}

We next provide an error analysis for En-POD solutions. First, we present several results obtained in \cite{GJS17}, which we  use in the analysis.

We first define the $L^2(\Omega)$ projection operator $\Pi_R$: $L^2(\Omega) \rightarrow X_R$ by
\begin{equation*}
(u-\Pi_R u , \varphi)=0\qquad \forall \varphi \in X_R.
\end{equation*}
\begin{lemma} \label{lm:L2err} 
	{\rm[$L^2(\Omega)$ norm of the error between snapshots and their projections onto the POD space]}  We have
	\begin{equation*}
	\frac{1}{J_S(N_S+1)} \sum_{j=1}^{J_S} \sum_{m=0}^{N_S}\Big \|  u_{h,S}^{j,m}-\sum_{i=1}^R (u_{h,S}^{j,m}, \varphi_i)\varphi_i\Big \| ^2 = \sum_{i=R+1}^{J_S(N_S+1)} \lambda_i 
	\end{equation*}
	and thus for $j=1,\ldots,J_S$,
	\begin{equation*}
	\frac{1}{N_S+1} \sum_{m=0}^{N_S}\Big \|  u_{h,S}^{j,m}-\sum_{i=1}^R (u_{h,S}^{j,m}, \varphi_i)\varphi_i\Big \| ^2 \leq J_S\sum_{i=R+1}^{J_S(N_S+1)} \lambda_i .
	\end{equation*}
\end{lemma}

\begin{lemma}\label{lm:H1err}{\rm [$H^1(\Omega)$ norm of the error between snapshots and their projections in the POD space.]}  We have
	\begin{equation*}
	\frac{1}{J_S(N_S+1)} \sum_{j=1}^{J_S} \sum_{m=0}^{N_S}\Big \| \nabla \Big( u_{h,S}^{j,m}-\sum_{i=1}^R (u_{h,S}^{j,m}, \varphi_i)\varphi_i\Big)\Big\| ^2 
	=\sum_{i=R+1}^{J_S(N_S+1)} \lambda_i  \|  \nabla \varphi_i\|^2
	\end{equation*}
	and thus, for $j=1,\ldots,J_S$,
	\begin{equation*}
	\frac{1}{N_S+1} \sum_{m=0}^{N_S} \Big\| \nabla\Big( u_{h,S}^{j,m}-\sum_{i=1}^R (u_{h,S}^{j,m}, \varphi_i)\varphi_i \Big)\Big\| ^2 \leq J_S\sum_{i=R+1}^{J_S(N_S+1)} \lambda_i  \|  \nabla \varphi_i\|^2.
	\end{equation*}
\end{lemma}

\begin{lemma}\label{lm:Projerr}{\rm [Error in the projection onto the POD space]}
	Consider the partition $0=t_0<t_1< \cdots < t_{N_S} = T$ used in Section \ref{POD_sec}. 
	For any $u \in H^{1}(0,T;[H^{s+1}(\Omega)]^d)$, let $u^{m}=u(\cdot, t_m)$. Then, the error in the projection onto the POD space $X_R$ satisfies the estimates
	\begin{equation*}
	\begin{aligned}
	\frac{1}{N_S+1}& \sum_{m=0}^{N_S} \|  u^{m}-\Pi_R u^m\| ^2
	\\&
	\leq \inf_{j\in\{1,\ldots,J_S\}}\frac{2}{N_S+1} \sum_{m=0}^{N_S} \|  u^{m}-u^{j,m}_S\| ^2+C\left( h^{2s+2} + \triangle t^4  \right)+ 2J_S\sum_{i=R+1}^{J_S(N_S+1)} \lambda_i 
	\end{aligned}
	\end{equation*}
	\begin{equation*}
	\begin{aligned}
	\frac{1}{N_S+1} &\sum_{m=0}^{N_S} \|  \nabla \left(u^{m}-\Pi_R u^m\right)\| ^2
	\\&
	\leq \inf_{{j\in\{1,\ldots,J_S\}}}\frac{2}{N_S+1} \sum_{m=0}^{N_S} \left(\|  \nabla (u^{m}-u_S^{j,m})\| ^2+\|\hspace{-1pt}|  {\mathbb S}_R \|\hspace{-1pt}| _2 \|   u^{m}-u_S^{j,m}\| ^2  \right)
	\\&
	+(C+h^2 \|\hspace{-1pt}|  {\mathbb S}_R \|\hspace{-1pt}| _2 ) h^{2s} + (C+\|\hspace{-1pt}|  {\mathbb S}_R \|\hspace{-1pt}| _2 )\triangle t^4  
	+ 2J_S\sum_{i=R+1}^{J_S(N_S+1)} \|  \nabla \varphi_i\| ^2\lambda_i . 
	\end{aligned}
	\end{equation*}
\end{lemma}

To bound the error between the POD approximations and the true solutions, we assume
the following regularity for the true solutions and body forces:
\begin{gather*}
u^{j} \in L^{\infty}(0,T;H^{s+1}(\Omega))\cap H^{1}(0,T;H^{s+1}(\Omega))\cap
H^{2}(0,T;L^{2}(\Omega)),\\
p^{j} \in L^{2}(0,T;H^{s}(\Omega)),\quad \text{and}\quad f^{j} \in L^{2}%
(0,T;L^{2}(\Omega)).
\end{gather*}

We assume the following estimate is also valid as done in \cite{IW14}.
\begin{assumption}\label{assumption1}
	Consider the partition $0=t_0<t_1< \cdots < t_{N_S} = T$ used in Section \ref{POD_sec}. 
	For any $u \in H^{1}(0,T;[H^{s+1}(\Omega)]^d)$, let $u^{m}=u(\cdot, t_m)$. Then, the error in the projection onto the POD space $X_R$ satisfies the estimates
	
	\begin{equation*}
	\begin{aligned}
	& \|  \nabla \left(u^{m}-\Pi_R u^m\right)\| ^2
	\\&
	\leq \inf_{{j\in\{1,\ldots,J_S\}}}\frac{2}{N_S+1} \sum_{m=0}^{N_S} \left(\|  \nabla (u^{m}-u_S^{j,m})\| ^2+\|\hspace{-1pt}|  {\mathbb S}_R \|\hspace{-1pt}| _2 \|   u^{m}-u_S^{j,m}\| ^2  \right)
	\\&
	+(C+h^2 \|\hspace{-1pt}|  {\mathbb S}_R \|\hspace{-1pt}| _2 ) h^{2s} + (C+\|\hspace{-1pt}|  {\mathbb S}_R \|\hspace{-1pt}| _2 )\triangle t^4  
	+ 2J_S\sum_{i=R+1}^{J_S(N_S+1)} \|  \nabla \varphi_i\| ^2\lambda_i . 
	\end{aligned}
	\end{equation*}
\end{assumption}

Let $e^{j,n}=u^{j,n}-u_{R}^{j,n}$ denote the error between the true solution
and the POD approximation; then, we have the following error estimates.

\begin{theorem}
	[Error analysis of En-POD]\label{th:errEn-POD} Consider the
	method \eqref{EnB-POD-Weak} and the partition $0=t_0<t_1< \cdots <t_{N_S}=T $ used in Section \ref{POD_sec}. Suppose that for any $0\leq n\leq N_S$, the following conditions hold
	\begin{gather}
	\frac{C\triangle t \|\hspace{-1pt}|  {\mathbb S}_R \|\hspace{-1pt}| _2^{1/2}}{\nu}\| \nabla
	(2u_{R}^{ j,n}-u_R^{j,n-1}-<u_R>^n)\| ^{2} <1\text{ , \qquad} j=1,...,J. 
	\label{cond:err}
	\end{gather}
	Then, for any $1\leq N\leq N_S$, there is a positive constant $C$ such that
	\begin{equation*}
	\begin{aligned}
	\frac{1}{2}&\|  e^{j,N}\| ^{2}+\frac{\Delta t}{32}\sum_{n=1}^{N-1}\nu\| \nabla e^{j,n+1}\| ^{2}\\
	&\leq C\Bigg (\Delta t^4 + \|\hspace{-1pt}|  {\mathbb S}_R \|\hspace{-1pt}| _2^{-1/2} \Delta t^3+ \|\hspace{-1pt}|  {\mathbb S}_R \|\hspace{-1pt}| _2 \Delta t^4+ \|\hspace{-1pt}|  {\mathbb S}_R \|\hspace{-1pt}| _2 h^{2s+2}\\
	&\quad+  \|\hspace{-1pt}|  {\mathbb S}_R \|\hspace{-1pt}| _2^{-1/2} h^{2s}\Delta t^{-1}+ \|\hspace{-1pt}|  {\mathbb S}_R \|\hspace{-1pt}| _2^{1/2} h^{2s+2}\Delta t^{-1}
	+  \|\hspace{-1pt}|  {\mathbb S}_R \|\hspace{-1pt}| _2^{1/2} \Delta t^3+h^{2s}\\
	& \quad+(1+N_S\ \|\hspace{-1pt}|  {\mathbb S}_R \|\hspace{-1pt}| _2^{-1/2}) \Big(\inf_{{j\in\{1,\ldots,J_S\}}}\frac{1}{N_S} \sum_{m=1}^{N_S} (\|  \nabla (u^{m}-u_S^{j,m})\| ^2\\
	&\quad+\|\hspace{-1pt}|  {\mathbb S}_R \|\hspace{-1pt}| _2 \|   u^{m}-u_S^{j,m}\| ^2  )+ J_S\sum_{i=R+1}^{J_SN_S} \|  \nabla \varphi_i\| ^2\lambda_i  \Big)\Bigg).
	\end{aligned}
	\end{equation*}
\end{theorem}

{\allowdisplaybreaks

	\begin{proof}
		For $j=1,\cdots, J$, the true solutions $u^{j}$ of the Navier-Stokes equations   satisfies
		\begin{equation}\label{eq:convtrue}
		\begin{aligned}
		&(\frac{3u^{j,n+1}-4u^{j,n}+u^{j,n-1}}{2\Delta t}, \varphi)+ b^{*}(u^{j,n+1},
		u^{j,n+1}, \varphi)+ \nu(\nabla u^{j, n+1}, \nabla v_{h})\\
		&\quad- (p^{j, n+1},\nabla\cdot \varphi)
		=(f^{j,n+1}, \varphi)+ Intp(u^{j,n+1};\varphi)\qquad\text{for any }
		\varphi\in X_{R},
		\end{aligned}
		\end{equation}
		where $Intp(u^{j,n+1};\varphi)$ is defined as
		\begin{equation*}
		Intp(u^{j,n+1};\varphi)=(\frac{3u^{j,n+1}-4u^{j,n}+u^{j,n-1}}{2\Delta t}-u_{t}^j
		(t_{n+1}),\varphi).
		\end{equation*}

		\noindent Let
		\begin{equation}
		e^{j,n}=u^{j,n}-u_{R}^{j,n}=(u^{j,n}-\Pi_{R} u^{j,n})+(\Pi_{R} u^{j,n}-u_{R}^{j,n})=\eta^{j,n}+\xi_{R}^{j,n} \text{ ,\qquad} j=1,...,J\text{ ,}\nonumber
		\end{equation}

		\noindent where $\Pi_{R} u_{j}^{n} \in X_{R} $ is the $L^2$ projection of $u^{j,n}$
		in $X_{R}.$ Subtracting (\ref{EnB-POD-Weak}) from (\ref{eq:convtrue}) gives
		
		\begin{gather}
		(\frac{3\xi_{R}^{j,n+1}-4\xi_{R}^{j,n}+\xi_{R}^{j,n-1}}{ 2\Delta t},\varphi) +\nu(\nabla\xi
		_{R}^{j,n+1},\nabla \varphi)+b^{*}(u^{j,n+1},u^{j,n+1},\varphi)\nonumber\\
		-b^{*}(2u_R^{j,n}-u_R^{j,n-1}-<u_R>^n,2u_{R}^{j,n}-u_R^{j,n-1},\varphi)\nonumber\\
		-b^{*}(<u_R>^n,u_{R}^{j,n+1},\varphi) -(p^{j,n+1},\nabla\cdot \varphi)\label{eq:err}\\
		=-(\frac{3\eta^{j,n+1}-4\eta^{j,n}+\eta^{j,n-1}}{2\Delta t},\varphi) -\nu(\nabla\eta^{j,n+1},\nabla \varphi) +Intp(u^{j,n+1};\varphi)\text{ .}\nonumber
		\end{gather}
		Set $\varphi=\xi_{R}^{j,n+1}\in X_{R}$ and rearrange the nonlinear
		terms. By the definition of the $L^2$ projection, we have $(3\eta^{j,n+1}-4\eta^{j,n}+\eta^{j,n-1}, \xi_{R}^{j,n+1})=0$. Thus \eqref{eq:err} becomes
		\begin{gather}
		\frac{1}{4\Delta t}\left(\|\xi_{R}^{j,n+1}\|^{2}+\|2\xi_{R}^{j,n+1}-\xi_{R}^{j,n}\|^{2}\right)-\frac{1}{4\Delta t}\left(\|\xi
		_{R}^{j,n}\|^{2}+\|2\xi
		_{R}^{j,n}-\xi
		_{R}^{j,n-1}\|^{2}\right)\nonumber\\
		+\frac{1}{4\Delta t}\|\xi_{R}^{j,n+1}-2\xi_{R}^{j,n}-\xi_{R}^{j,n-1}\|^{2}+\nu
		\|\nabla\xi_{R}^{j,n+1}\|^{2}\nonumber\\
		=-b^{*}(u^{j,n+1},u^{j,n+1},\xi_{R}^{j,n+1})+b^{*}(2u_{R}^{j,n}-u_{R}^{j,n-1}
		,u_{R}^{j,n+1},\xi_{R}^{j,n+1})\label{eq:err1}\\
		-b^{*}(2u_R^{j,n}-u_R^{j,n-1}-<u_R>^n,u_{R}^{j,n+1}-2u_{R}^{j,n}+u_{R}^{j,n-1},\xi_{R}^{j,n+1})\nonumber\\
		+(p^{j,n+1},\nabla\cdot\xi_{R}^{j,n+1})-\nu
		(\nabla\eta^{j,n+1},\nabla\xi_{R}^{j,n+1})+Intp(u^{j,n+1};\xi_{R}
		^{j,n+1}).\nonumber
		\end{gather}
		We rewrite the first two nonlinear terms as
		\begin{align*}\label{eq:nonlinear}
		&-b^{*}(u^{j,n+1},u^{j,n+1},\xi_{R}^{j,n+1})+b^{*}(2u_{R}^{j,n}-u_{R}^{j,n-1}
		,u_{R}^{j,n+1},\xi_{R}^{j,n+1})\nonumber\\
		&=-b^{*}(2e^{j,n}-e^{j,n-1},u^{j,n+1},\xi_{R}^{j,n+1})
		-b^{*}(2u_R^{j,n}-u_R^{j,n-1}
		,e^{j,n+1},\xi_{R}^{j,n+1})\nonumber\\
		&\qquad -b^{*}(u^{j,n+1}-2u^{j,n}+u^{j,n-1}
		,u^{j,n+1},\xi_{R}^{j,n+1})\\
		&=-b^{*}(2\eta^{j,n}-\eta^{j,n-1},u^{j,n+1},\xi_{R}^{j,n+1})-b^{*}(2\xi_R^{j,n}-\xi_R^{j,n-1},u^{j,n+1},\xi_{R}^{j,n+1})\nonumber\\
		&\qquad-b^{*}(2u_R^{j,n}-u_R^{j,n-1}
		,\eta^{j,n+1},\xi_{R}^{j,n+1})-b^{*}(u^{j,n+1}-2u^{j,n}+u^{j,n-1}
		,u^{j,n+1},\xi_{R}^{j,n+1}).\nonumber
		\end{align*}
		With the assumption that $u^j \in L^{\infty}(0,T; H^1(\Omega))$, we estimate the nonlinear terms as follows
		\begin{equation*}
		\begin{aligned}
		&b^{*}(u^{j,n+1}-2u^{j,n}+u^{j,n-1},u^{j,n+1},\xi_{R}^{j,n+1})\\ 
		&\leq C\|\nabla
		(u^{j,n+1}-2u^{j,n}+u^{j,n-1})\|||\nabla u^{j,n+1}\|||\nabla\xi_{R}
		^{j,n+1}\|\\
		&\leq\frac{\nu}{64}\|\nabla\xi_{R}^{j,n+1}\|^{2}+C\nu^{-1}\|\nabla(u
		^{j,n+1}-2u^{j,n}+u^{j,n-1})\|^{2}\\
		&\leq\frac{\nu}{64}\|\nabla\xi_{R}^{j,n+1}\|^{2}+\frac{C\Delta t^3}{\nu}
		(\int_{t^{n-1}}^{t^{n+1}}\| \nabla u_{tt}^j \|^{2} dt)
		\end{aligned}
		\end{equation*}
		and
		\begin{equation*}
		\begin{aligned}
		b^{*}(2&\eta^{j,n}-\eta^{j,n-1}, u^{j,n+1}, \xi_{R}^{j,n+1}) \\
		&\leq C\|\nabla (2\eta^{j,n}-\eta^{j,n-1})
		\|||\nabla u^{j,n+1}\|||\nabla\xi_{R}^{j,n+1}\|\\
		&\leq\frac{\nu}{64}\|\nabla\xi_{R}^{j,n+1}\|^{2}+C\nu^{-1}(\|\nabla\eta^{j,n}
		\|^{2}\|+\|\nabla\eta^{j,n-1}
		\|^{2}) .
		\end{aligned}
		\end{equation*}
		Using Young's inequality, inequality \eqref{In2}, and the result \eqref{Stability} from the stability analysis, i.e., $\Vert u_R^{j,n}\Vert^2\leq C$, we have
		\begin{equation*}
		\begin{aligned}
		b^{*}(2u_R^{j,n}-&u_R^{j,n-1},\eta^{j,n+1},\xi_{R}^{j,n+1}) \\
		&\leq C\|\nabla (2u_R
		^{j,n+1}-u_R^{j,n})\|^{1/2}\|2u_R
		^{j,n+1}-u_R^{j,n}\|^{1/2}\|\nabla\eta^{j,n+1}\|||\nabla\xi_{R}^{j,n+1}\|\\
		&\leq\frac{\nu}{64}\|\nabla\xi_{R}^{j,n+1}\|^{2}+C\nu^{-1}\|\nabla (2 u_R
		^{j,n}-u_R^{j,n-1})\|||\nabla\eta^{j,n+1}\|^{2}\text{ ,}
		\end{aligned}
		\end{equation*}
		Using inequality \eqref{In2}, Young's inequality, and $u^j \in L^{\infty}(0,T; H^1(\Omega))$, we have
		\begin{equation*}
		\begin{aligned}
		-2b^{*}(\xi^{j,n}_{R}, &u^{j,n+1}, \xi_{R}^{j,n+1}) \\
		&\leq C\| \nabla\xi
		^{j,n}_{R} \|^{1/2} \| \xi^{j,n}_{R}\| ^{1/2}\|\nabla
		u^{j,n+1}\|||\nabla\xi_{R}^{j,n+1}\|\\
		&\leq C\| \nabla\xi^{j,n}_{R} \|^{1/2} \| \xi^{j,n}_{R}\| 
		^{1/2}\|\nabla\xi_{R}^{j,n+1}\|\\
		&\leq C\left(\epsilon\|\nabla\xi_{R}^{j,n+1}\|^{2}+\frac{1}{\epsilon}\|\nabla\xi
		^{j,n}_{R} \|||\xi^{j,n}_{R} \|\right)\\
		&\leq C\left(\epsilon\|\nabla\xi_{R}^{j,n+1}\|^{2}+\frac{1}{\epsilon}\left(\delta
		\|\nabla\xi^{j,n}_{R} \|^{2}+\frac{1}{\delta}\|\xi^{j,n}_{R} \|\right)\right)\\
		&\leq \left(\frac{\nu}{64} \|\nabla\xi_{R}^{j,n+1}\|^{2}+\frac{\nu}{16} \|\nabla
		\xi^{j,n}_{R} \|^{2}\right)+\frac{C}{\nu^{3}}\|\xi^{j,n}_{R} \|^{2} ,
		\end{aligned}
		\end{equation*}
		Similarly,
		\begin{equation*}
		\begin{aligned}
		b^{*}(\xi^{j,n-1}_{R}, &u^{j,n+1}, \xi_{R}^{j,n+1}) \\
		&\leq C\| \nabla\xi
		^{j,n-1}_{R} \|^{1/2} \| \xi^{j,n-1}_{R}\| ^{1/2}\|\nabla
		u^{j,n+1}\|||\nabla\xi_{R}^{j,n+1}\|\\
		&\leq C\| \nabla\xi^{j,n-1}_{R} \|^{1/2} \| \xi^{j,n-1}_{R}\| 
		^{1/2}\|\nabla\xi_{R}^{j,n+1}\|\\
		&\leq C\left(\epsilon\|\nabla\xi_{R}^{j,n+1}\|^{2}+\frac{1}{\epsilon}\|\nabla\xi
		^{j,n-1}_{R} \|||\xi^{j,n-1}_{R} \|\right)\\
		&\leq C\left(\epsilon\|\nabla\xi_{R}^{j,n+1}\|^{2}+\frac{1}{\epsilon}\left(\delta
		\|\nabla\xi^{j,n-1}_{R} \|^{2}+\frac{1}{\delta}\|\xi^{j,n-1}_{R} \|\right)\right)\\
		&\leq \left(\frac{\nu}{64} \|\nabla\xi_{R}^{j,n+1}\|^{2}+\frac{\nu}{16} \|\nabla
		\xi^{j,n-1}_{R} \|^{2}\right)+\frac{C}{\nu^{3}}\|\xi^{j,n-1}_{R} \|^{2} .
		\end{aligned}
		\end{equation*}
		Next, we rewrite the third nonlinear term on the right-hand side of \eqref{eq:err1} as
		\begin{align}
		&b^{*}(2u_R^{j,n}-u_R^{j,n-1}-<u_R>^n, u_R^{j,n+1}-2u_R^{j,n}+u_R^{j,n-1}, \xi_R^{j,n+1})\\
		&= - b^{*}(2u_R^{j,n}-u_R^{j,n-1}-<u_R>^n, e^{j,n+1}-2e^{j,n}+e^{j,n-1}, \xi_R^{j,n+1}) \nonumber\\
		&\quad+b^{*}(2u_R^{j,n}-u_R^{j,n-1}-<u_R>^n, u^{j,n+1}-2u^{j,n}+u^{j,n-1}, \xi_R^{j,n+1}) \nonumber\\
		&= - b^{*}(2u_R^{j,n}-u_R^{j,n-1}-<u_R>^n, \eta^{j,n+1}-2\eta^{j,n}+\eta^{j,n-1}, \xi_R^{j,n+1}) \nonumber\\
		&\quad+b^{*}(2u_R^{j,n}-u_R^{j,n-1}-<u_R>^n, 2\xi_R^{j,n}-\xi_R^{j,n-1}, \xi_R^{j,n+1}) \nonumber\\
		&\quad+b^{*}(2u_R^{j,n}-u_R^{j,n-1}-<u_R>^n, u^{j,n+1}-2u^{j,n}+u^{j,n-1}, \xi_R^{j,n+1}) .\nonumber
		\end{align}
		Then, we have the following estimates on the above terms:
		\begin{align*}
		&-b^{*}(2u_{R}^{j,n}-u_R^{j,n-1}-<u_R>^n,\eta^{j,n+1}-2\eta^{j,n}+\eta^{j,n-1},\xi_{R}^{j,n+1})\\ 
		&\leq C\|\nabla (2u_{R}^{j,n}-u_R^{j,n-1}-<u_R>^n)\|||\nabla \left(\eta^{j,n+1}-2\eta^{j,n}+\eta^{j,n-1}\right)\|||\nabla\xi_{R}^{j,n+1}\|\nonumber\\
		&\leq C\nu^{-1}\|\nabla (2u_{R}^{j,n}-u_R^{j,n-1}-<u_R>^n)\|^{2} \|\nabla \left(\eta^{j,n+1}-2\eta^{j,n}+\eta^{j,n-1}\right)\Vert^2\nonumber\\
		&\qquad\qquad +\frac{\nu}{64}\|\nabla\xi_{R}^{j,n+1}\|^{2}.\nonumber
		\end{align*}
		By skew symmetry, inequality \eqref{In1} and the inverse inequality, we have
		\begin{align*}
		&b^{*}(2u_{R}^{j,n}-u_R^{j,n-1}-<u_R>^n,2\xi_R^{j,n}-\xi_R^{j,n-1},\xi_{R}^{j,n+1})\\
		&= b^{*}(2u_{R}^{j,n}-u_R^{j,n-1}-<u_R>^n, \xi_{R}^{j,n+1}, \xi_R^{j,n+1}-2\xi_R^{j,n}+\xi_R^{j,n-1})\nonumber\\
		&\leq \Vert \nabla (2u_{R}^{j,n}-u_R^{j,n-1}-<u_R>^n)\Vert\Vert \nabla \xi_{R}^{j,n+1}\Vert \|\hspace{-1pt}|  {\mathbb S}_R \|\hspace{-1pt}| _2^{1/4}\Vert \xi_R^{j,n+1}-2\xi_R^{j,n}+\xi_R^{j,n-1}
		\Vert\nonumber\\ 
		&\leq \frac{1}{8\Delta t}\Vert \xi_R^{j,n+1}-2\xi_R^{j,n}+\xi_R^{j,n-1}
		\Vert^2\nonumber\\
		&\qquad+C\Delta t\|\hspace{-1pt}|  {\mathbb S}_R \|\hspace{-1pt}| _2^{1/2}\Vert \nabla (2u_{R}^{j,n}-u_R^{j,n-1}-<u_R>^n)\Vert^2\Vert \nabla \xi_{R}^{j,n+1}\Vert^2.\nonumber
		\end{align*}
		For the last nonlinear term we have
		\begin{align*}
		&b^{*}(2u_{R}^{j,n}-u_R^{j,n-1}-<u_R>^n,u^{j,n+1}-2u^{j,n}+u^{j,n-1},\xi_{R}^{j,n+1})\\
		&\leq C\|\nabla
		(2u_{R}^{j,n}-u_R^{j,n-1}-<u_R>^n)\|||\nabla \left( u^{j,n+1}-2u^{j,n}+u^{j,n-1}\right)\|||\nabla\xi_{R}^{j,n+1}
		\|\nonumber\\
		&\leq\frac{\nu}{64}\|\nabla\xi_{R}^{j,n+1}\|^{2}+\frac{C \Delta t^3}{\nu}\|\nabla
		(2u_{R}^{j,n}-u_R^{j,n-1}-<u_R>^n)\|^{2}\left(\int_{t_{n-1}}^{t_{n+1}}\| \nabla u_{tt}^j\|^{2} \text{
		}dt\right)\text{ .}\nonumber
		\end{align*}
		Next, consider the pressure term. Since $\xi_{R}^{j,n+1}\in X_{R}\subset V_h$
		we have for $q_{h}^{n+1}\in Q_h$
		\begin{equation*}
		\begin{aligned}
		(p^{j,n+1},\nabla\cdot\xi_{R}^{j,n+1})&=(p^{j,n+1}-q_{h}^{n+1},
		\nabla\cdot\xi_{R}^{j,n+1})\\
		&\leq\|p^{j,n+1}-q_{h}^{n+1}\|||\nabla\cdot\xi_{R}^{j,n+1}\|\\
		&\leq\frac{\nu}{64}\|\nabla\xi_{R}^{j,n+1}\|^{2}+C \nu^{-1}\|p
		^{j,n+1}-q_{h}^{n+1}\|^{2} \text{ .}
		\end{aligned}
		\end{equation*}
		The viscous term is bounded as
		\begin{equation*}
		\begin{aligned}
		\nu(\nabla\eta^{j,n+1},\nabla\xi_{R}^{j,n+1}) &\leq\nu\|\nabla\eta
		^{j,n+1}\| \|\nabla\xi_{R}^{j,n+1}\|\\
		&\leq C\nu\|\nabla\eta^{j,n+1}\|^{2}+ \frac{\nu}{64}\|\nabla\xi_{R}
		^{j,n+1}\|^{2} \text{ .}
		\end{aligned}
		\end{equation*}
		Finally,
		\begin{align*}
		Intp(u^{j,n+1};\xi_{R}^{j,n+1})&=\left(\frac{3u^{j,n+1}-4u^{j,n}
			+u^{j,n-1}}{2\Delta
			t}-u_{t}^j(t_{n+1}),\xi_{R}^{j,n+1}\right)\\
		&\leq C\|\frac{3u^{j,n+1}-4u^{j,n}+u^{j,n-1}}{2\Delta t}-u_{t}^j(t_{n+1})\| \|\nabla
		\xi_{R}^{j,n+1}\|\nonumber\\
		&\leq\frac{\nu}{64}\|\nabla\xi_{R}^{j,n+1}\|^{2}+\frac{C}{\nu}\|\frac
		{3u^{j,n+1}-4u^{j,n}+u^{j,n-1}}{2\Delta t}-u_{t}^j(t_{n+1})\|^{2}\nonumber\\
		&\leq\frac{\nu}{64}\|\nabla\xi_{R}^{j,n+1}\|^{2}+\frac{C\Delta t^3}{\nu}
		\int_{t_{n}}^{t_{n+1}}\|u_{ttt}^j\|^{2} dt \text{ .}\nonumber
		\end{align*}
		Combining, we now have the inequality
		\begin{align}\label{eq:err2}
		&\frac{1}{4\Delta t}\left(\|\xi_{R}^{j,n+1}\|^{2}+\|2\xi_{R}^{j,n+1}-\xi_R^{j,n}\|^{2}\right)-\frac{1}{4\Delta t}\left(\|\xi_{R}^{j,n}\|^{2}+\|2\xi_{R}^{j,n}-\xi_R^{j,n-1}\|^{2}\right)\nonumber\\
		&+\frac{1}{8\Delta t}\| \xi_{R}^{j,n+1}-2\xi_{R}^{j,n}+\xi_R^{j,n-1}\| ^{2}
		)+\frac{\nu}{64}\| \nabla\xi_{R}^{j,n+1}\| ^{2}+\frac{\nu}{16}\left(\| \nabla\xi_{R}^{j,n+1}\| ^{2}-\| \nabla\xi_{R}
		^{j,n}\| ^{2}\right)\nonumber\\
		&+\frac{\nu}{16}\left(\left(\| \nabla\xi_{R}^{j,n+1}\| ^{2}+\| \nabla\xi_{R}
		^{j,n}\| ^{2}\right)-\left(\| \nabla\xi_{R}^{j,n}\| ^{2}+\| \nabla\xi_{R}
		^{j,n-1}\| ^{2}\right)\right)\nonumber\\
		&+\left(\frac{\nu}{4}-C\triangle t \|\hspace{-1pt}|  {\mathbb S}_R \|\hspace{-1pt}| _2^{1/2}\| \nabla (2u_R^{j,n}-u_R^{j,n-1}-<u_R>^n)\| ^{2}\right)\| \nabla
		\xi_{R}^{j,n+1}\| ^{2}\nonumber\\
		&\leq C\nu^{-3}\left( \Vert \xi_R^{j,n}\Vert^2+\Vert \xi_R^{j,n-1}\Vert^2\right) +\frac{C\Delta t^3}{\nu}\left(\int_{t_{n-1}}^{t_{n+1}}\| \nabla u_{tt}^j
		\| ^{2}dt\right) \\
		&+C\nu^{-1}\left(\| \nabla\eta^{j,n}\| ^{2}+\| \nabla\eta^{j,n-1}\| ^{2}\right)+ C\nu^{-1}\| \nabla (2u_R^{j,n}-u_R^{j,n-1})\| \| \nabla\eta^{j,n+1}
		\| ^{2}\nonumber\\
		&+C\nu^{-1}\| \nabla
		(2u_R^{j,n}-u_R^{j,n-1}-<u_R>^n)\| ^{2}\|\nabla \left(\eta^{j,n+1}-2\eta^{j,n}+\eta^{j,n-1}\right)\Vert^2\nonumber\\
		&+\frac{C\Delta t^3}{\nu}\| \nabla (2u_R^{j,n}-u_R^{j,n-1}-<u_R>^n)\| ^{2}\left(\int_{t_{n-1}}^{t_{n+1}
		}\| \nabla u_{tt}^j\| ^{2}dt\right)\nonumber\\
		&+C\nu^{-1}\|  p^{j,n+1}-q_{h}^{n+1}
		\| ^{2}+C\nu\| \nabla\eta^{j,n+1}\| ^{2}+\frac{C\Delta t^3}{\nu}\left(\int_{t_{n-1}
		}^{t_{n+1}}\|  u_{ttt}^j\| ^{2}dt\right)\text{ .}\nonumber
		\end{align}
		By the time-step condition \eqref{cond:err}, we have $\frac{\nu}{4}-C\triangle t \|\hspace{-1pt}|  {\mathbb S}_R \|\hspace{-1pt}| _2^{1/2}\| \nabla
		(2u_{R}^{ j, n}-u_R^{j,n-1}-<u_R>^n)\| ^{2}>0$ and thus can be removed from the left hand side of the inequality. Taking the sum of (\ref{eq:err2}) from $n=1$ to
		$n=N-1$ and multiplying through by $2\Delta t$, we obtain
		\begin{align*}
		&\frac{1}{2}\|\xi_{R}^{j,N}\|^{2}+\frac{1}{2}\|2\xi_{R}^{j,N}-\xi_R^{j,N-1}\|^{2}+\frac{1}{4}\sum_{n=1}^{N-1}\| \xi_{R}^{j,n+1}-2\xi_{R}^{j,n}+\xi_{R}^{j,n-1}
		\| ^{2}\\
		&+\frac{\nu \Delta t}{32}\sum_{n=1}^{N-1}\Vert \nabla\xi_R^{j,n+1}\Vert^2+\frac{\nu\Delta t}{4}\|\nabla \xi_{R}
		^{j,N}\|^{2}+\frac{\nu\Delta t}{8}\|\nabla \xi_{R}
		^{j,N-1}\|^{2}\nonumber\\
		&\leq\frac{1}{2}\|\xi_{R}^{j,1}\|^{2}+\frac{1}{2}\|2\xi_{R}^{j,1}-\xi_R^{j,0}\|^{2}+\frac{\nu\Delta t}{4}\|\nabla\xi
		_{R}^{j,1}\|^{2}+\frac{\nu\Delta t}{8}\|\nabla\xi
		_{R}^{j,0}\|^{2}\nonumber\\
		&+\Delta t\sum_{n=1}^{N-1}\frac{C}{\nu^{3}}\| \xi_{R}
		^{j,n}\| ^{2}
		+\Delta t\sum_{n=1}^{N-1}\Bigg\{  \frac{C\Delta t^3}{\nu}\left(\int_{t_{n-1}}^{t_{n+1}}\| \nabla u_{tt}^j
		\| ^{2}dt\right) \\
		&+C\nu^{-1}\left(\| \nabla\eta^{j,n}\| ^{2}+\| \nabla\eta^{j,n-1}\| ^{2}\right)+ C\nu^{-1}\| \nabla (2u_R^{j,n}-u_R^{j,n-1})\| \| \nabla\eta^{j,n+1}
		\| ^{2}\nonumber\\
		&+C\nu^{-1}\| \nabla
		(2u_R^{j,n}-u_R^{j,n-1}-<u_R>^n)\| ^{2}\|\nabla \left(\eta^{j,n+1}-2\eta^{j,n}+\eta^{j,n-1}\right)\Vert^2\nonumber\\
		&+\frac{C\Delta t^3}{\nu}\| \nabla (2u_R^{j,n}-u_R^{j,n-1}-<u_R>^n)\| ^{2}\left(\int_{t_{n-1}}^{t_{n+1}
		}\| \nabla u_{tt}^j\| ^{2}dt\right)\nonumber\\
		&+C\nu^{-1}\|  p^{j,n+1}-q_{h}^{n+1}
		\| ^{2}+C\nu\| \nabla\eta^{j,n+1}\| ^{2}+\frac{C\Delta t^3}{\nu}\left(\int_{t_{n-1}
		}^{t_{n+1}}\|  u_{ttt}^j\| ^{2}dt\right)\Bigg\}\text{ .}\nonumber
		\end{align*}
		Using the stability result, i.e. $\Delta t \sum_{n=1}^{N-1} \nu \Vert \nabla u_R^{j,n+1}\Vert^2 \leq C$ and Assumption \eqref{assumption1}, we have
		\begin{align*}
		& C\nu^{-1} \Delta t \sum_{n=1}^{N-1}\Vert \nabla (2 u_R^{j,n}-u_R^{j,n-1})\Vert \Vert \nabla \eta^{j,n+1}\Vert^2\\
		&\qquad \leq C \nu^{-2} \Bigg( \inf_{{j\in\{1,\ldots,J_S\}}}\frac{2}{N_S+1} \sum_{m=0}^{N_S} \left(\|  \nabla (u^{m}-u_S^{j,m})\| ^2+\|\hspace{-1pt}|  {\mathbb S}_R \|\hspace{-1pt}| _2 \|   u^{m}-u_S^{j,m}\| ^2  \right)
		\nonumber\\
		&\qquad\quad
		+(C+h^2 \|\hspace{-1pt}|  {\mathbb S}_R \|\hspace{-1pt}| _2 ) h^{2s} + (C+\|\hspace{-1pt}|  {\mathbb S}_R \|\hspace{-1pt}| _2 )\triangle t^4  
		+ 2J_S\sum_{i=R+1}^{J_S(N_S+1)} \|  \nabla \varphi_i\| ^2\lambda_i\Bigg ).\nonumber
		\end{align*}
		Because $u_R^{j,0}=\sum_{i=1}^R (u^{j,0}, {\varphi}_i)\varphi_i$, we have $\|  \xi_{R}^{j,0}\| ^2=0$ and  $\|  \nabla \xi_{R}^{j,0}\| ^2=0$. Now applying Lemma \ref{lm:Projerr} gives
		\begin{align*}
		&\frac{1}{2}\|\xi_{R}^{j,N}\|^{2}+\frac{1}{2}\|2\xi_{R}^{j,N}-\xi_R^{j,N-1}\|^{2}+\frac{1}{4}\sum_{n=1}^{N-1}\| \xi_{R}^{j,n+1}-2\xi_{R}^{j,n}+\xi_{R}^{j,n-1}
		\| ^{2}\\
		&+\frac{\nu \Delta t}{32}\sum_{n=1}^{N-1}\Vert \nabla\xi_R^{j,n+1}\Vert^2+\frac{\nu\Delta t}{4}\|\nabla \xi_{R}
		^{j,N}\|^{2}+\frac{\nu\Delta t}{8}\|\nabla \xi_{R}
		^{j,N-1}\|^{2}\nonumber\\
		&\leq \Delta t\sum_{n=0}^{N-1}\frac{C}{\nu^{2}}\| \xi_{R}
		^{j,n}\| ^{2}+\frac{C\Delta t^{4}}{\nu}\| |\nabla u_{tt}^j||| _{2,0}^{2}+\Bigg(C\nu^{-2}+C N_S\Delta t\nu^{-1}\nonumber\\
		&\quad+CN_S\|\hspace{-1pt}|  {\mathbb S}_R \|\hspace{-1pt}| _2^{-1/2}
		+CN_S\nu \triangle t \Bigg) \cdot \Bigg( \inf_{{j\in\{1,\ldots,J_S\}}}\frac{1}{N_S} \sum_{m=1}^{N_S} (\|  \nabla (u^{m}-u_S^{j,m})\| ^2\nonumber\\
		&\quad+\|\hspace{-1pt}|  {\mathbb S}_R \|\hspace{-1pt}| _2 \|   u^{m}-u_S^{j,m}\| ^2  )
		+(C+h^2 \|\hspace{-1pt}|  {\mathbb S}_R \|\hspace{-1pt}| _2 ) h^{2s} + (C+\|\hspace{-1pt}|  {\mathbb S}_R \|\hspace{-1pt}| _2 )\triangle t^4  \nonumber\\
		&\quad+ J_S\sum_{i=R+1}^{J_SN_S} \|  \nabla \varphi_i\| ^2\lambda_i  \Bigg )
		+C  \Delta t^3 \|\hspace{-1pt}|  {\mathbb S}_R \|\hspace{-1pt}| _2^{-1/2}\| |\nabla
		u_{tt}^j||| _{2,0}^{2}\nonumber\\
		&\quad+C\frac{h^{2s}}{\nu}\| |p^j||| _{2,s}^{2}
		+\frac{C{\Delta t}^{4}}{\nu
		}\| |u_{ttt}^j||| _{2,0}^{2}.\nonumber
		\end{align*}
		We use the discrete Gronwall
		inequality \cite[p. 176]{GR79} to obtain
		\begin{align}
		&\frac{1}{2}\|\xi_{R}^{j,N}\|^{2}+\frac{1}{2}\|2\xi_{R}^{j,N}-\xi_R^{j,N-1}\|^{2}+\frac{1}{4}\sum_{n=1}^{N-1}\| \xi_{R}^{j,n+1}-2\xi_{R}^{j,n}+\xi_{R}^{j,n-1}
		\| ^{2}\nonumber\\
		&+\frac{\nu \Delta t}{32}\sum_{n=1}^{N-1}\Vert \nabla\xi_R^{j,n+1}\Vert^2+\frac{\nu\Delta t}{4}\|\nabla \xi_{R}
		^{j,N}\|^{2}+\frac{\nu\Delta t}{8}\|\nabla \xi_{R}
		^{j,N-1}\|^{2}\nonumber\\
		&\leq \exp \Big(\frac{CT}{\nu^{2}}\Big)\Bigg\{\frac{C\Delta t^{4}}{\nu}\| |\nabla u_{tt}^j||| _{2,0}^{2}+\Bigg(C\nu^{-2}+C N_S\Delta t\nu^{-1}+CN_S\|\hspace{-1pt}|  {\mathbb S}_R \|\hspace{-1pt}| _2^{-1/2}\label{ineq:errlast1}\\
		&\quad
		+CN_S\nu \triangle t \Bigg) \cdot\Bigg( \inf_{{j\in\{1,\ldots,J_S\}}}\frac{1}{N_S} \sum_{m=1}^{N_S} (\|  \nabla (u^{m}-u_S^{j,m})\| ^2+\|\hspace{-1pt}|  {\mathbb S}_R \|\hspace{-1pt}| _2 \|   u^{m}-u_S^{j,m}\| ^2  )\nonumber\\
		&
		\quad+(C+h^2 \|\hspace{-1pt}|  {\mathbb S}_R \|\hspace{-1pt}| _2 ) h^{2s} + (C+\|\hspace{-1pt}|  {\mathbb S}_R \|\hspace{-1pt}| _2 )\triangle t^4  + J_S\sum_{i=R+1}^{J_SN_S} \|  \nabla \varphi_i\| ^2\lambda_i   \Bigg )\nonumber\\
		&
		\quad+C  \Delta t^3 \|\hspace{-1pt}|  {\mathbb S}_R \|\hspace{-1pt}| _2^{-1/2}\| |\nabla
		u_{tt}^j||| _{2,0}^{2}+C\frac{h^{2s}}{\nu}\| |p^j||| _{2,s}^{2}
		+\frac{C{\Delta t}^{4}}{\nu
		}\| |u_{ttt}^j||| _{2,0}^{2}\Bigg\}\text{ .}\nonumber
		\end{align}
		Recall that $e^{j,n}=\eta^{j,n}+\xi_{R}^{j,n}$. To simplify formulas, we drop the second, third, fifth and sixth term on the left hand side of \eqref{ineq:errlast1}. Then, by the triangle
		inequality and Lemma \ref{lm:Projerr}, absorbing constants, we have 
		\begin{align*}
		\frac{1}{2}&\|  e^{j,N}\| ^{2}+\frac{\Delta t}{32}\sum_{n=1}^{N-1}\nu\| \nabla e^{j,n+1}\| ^{2}\\
		&\leq C\Bigg (\Delta t^4 + \|\hspace{-1pt}|  {\mathbb S}_R \|\hspace{-1pt}| _2^{-1/2} \Delta t^3+ \|\hspace{-1pt}|  {\mathbb S}_R \|\hspace{-1pt}| _2 \Delta t^4+ \|\hspace{-1pt}|  {\mathbb S}_R \|\hspace{-1pt}| _2 h^{2s+2}\nonumber\\
		&\quad+  \|\hspace{-1pt}|  {\mathbb S}_R \|\hspace{-1pt}| _2^{-1/2} h^{2s}\Delta t^{-1}+ \|\hspace{-1pt}|  {\mathbb S}_R \|\hspace{-1pt}| _2^{1/2} h^{2s+2}\Delta t^{-1}
		+  \|\hspace{-1pt}|  {\mathbb S}_R \|\hspace{-1pt}| _2^{1/2} \Delta t^3+h^{2s}\nonumber\\
		& \quad+(1+N_S\ \|\hspace{-1pt}|  {\mathbb S}_R \|\hspace{-1pt}| _2^{-1/2}) \Big(\inf_{{j\in\{1,\ldots,J_S\}}}\frac{1}{N_S} \sum_{m=1}^{N_S} (\|  \nabla (u^{m}-u_S^{j,m})\| ^2\nonumber\\
		&\quad+\|\hspace{-1pt}|  {\mathbb S}_R \|\hspace{-1pt}| _2 \|   u^{m}-u_S^{j,m}\| ^2  )+ J_S\sum_{i=R+1}^{J_SN_S} \|  \nabla \varphi_i\| ^2\lambda_i  \Big)\Bigg).\nonumber
		\end{align*}
		This completes the proof.
	\end{proof}
}

\section{Numerical experiments}
\label{numex}
In this section, we examine the efficacy of the \emph{EnB-POD} method via the numerical simulation of
a flow between two offset circles. We focus on testing the accuracy of our methods compared to \emph{EnB-full}. 

{\color{red}The computational cost of the \emph{EnB-POD} algorithm can be split into two phases: an expensive offline phase and a cheaper online phase. In the offline phase, we assemble the snapshot matrix $\mathbb{A}$ and, as outlined in Section \ref{POD_sec}, use this system to generate the reduced basis. Then, in the online phase we use the reduced basis to assemble \eqref{EnB-POD-Weak} which we then solve. The major computational gain of \emph{EnB-POD} comes from the fact that it should be much cheaper to solve \eqref{EnB-POD-Weak} than the original system \eqref{EnB}. For a more detailed discussion about the computational efficiency of the ensemble-POD algorithm, see  \cite[Section 6.1]{GJS17}. The method we develop here has the advantage of higher-order accuracy, which allows larger time steps than that of the first-order algorithm studied in \cite{GJS17} which further reduces the computational cost which especially useful in long-time simulations.}

We consider the two-dimensional flow between two offset circles similar to the one investigated in \cite{GJS17}. The domain is a disk with a smaller off-center disc inside. Let $r_{1}=1$, $r_{2}=0.1$, $c_{1}=1/2$, and $c_{2}=0$; then, the domain is given by
\[
\Omega=\{(x,y):x^{2}+y^{2}\leq r_{1}^{2} \text{ and } (x-c_{1})^{2}%
+(y-c_{2})^{2}\geq r_{2}^{2}\}.
\]
For our test problems, we generate perturbed initial conditions by solving a steady Stokes problem with perturbed body forces given by 
\[
f_{\epsilon}(x,y,t)=f(x,y,t)+\epsilon\big(\sin(3\pi x)\sin(3\pi y),\cos(3\pi
x)\cos(3\pi y)\big)^{T}.
\]
The deterministic flow is driven by the counterclockwise rotational body force
\[
f(x,y,t)=\big(-4y(1-x^{2}-y^{2})\,,\,4x(1-x^{2}-y^{2})\big)^{T}.
\]
No-slip, no-penetration boundary conditions are imposed on both circles. All computations are done using the FEniCS software suite \cite{LNW12}. 
Different perturbations defined by varying $\epsilon$ are applied. We discretize in space via the $P^2$-$P^1$ Taylor-Hood element pair. The mesh utilized is generated using the FEniCS built-in \textbf{mshr} package resulting in 16,457 total degrees of freedom; see Figure \ref{meshCinC}. We set the viscosity coefficient $\nu = \frac{1}{50}$.

\begin{figure}[h!]
	\centering
	\includegraphics[width = 6cm]{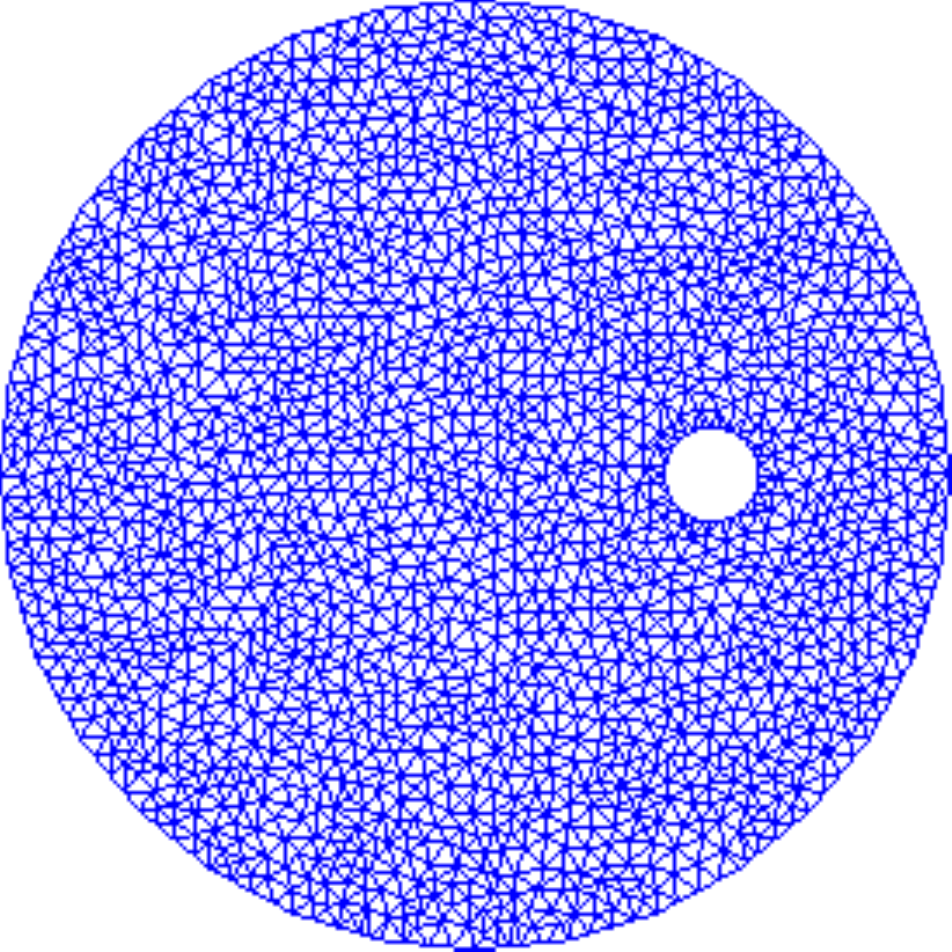} 
	\caption{Mesh for flow between offset circles resulting in 7,405 total degrees of freedom for the Taylor-Hood element pair.}
	\label{meshCinC}
\end{figure}

In order to generate the POD basis, we use two perturbations of the initial conditions corresponding to $\epsilon_{1} = 0.001$ and $\epsilon_{2} = -0.001$. Using a mesh that results in 16,457 total degrees of freedom and a fixed time step $\Delta t = .01$, we run the {\em EnB-full} using both perturbation from $t_{0} = 0$ to $T=5$, taking snapshots every $0.04$ seconds. In Figure \ref{eigvals_ex1_ex2}, we illustrate the decay of the singular values of the snapshot matrix. {\color{red}We note that by subtracting the mean snapshot from all the snapshots in the system $\mathbb{A}$ known as the centering trajectory approach \cite{HLB96} the numerical results may be improved. Specifically we expect the size of the eigenvalues (especially the first) to be smaller. In order to keep the numerical setting as close as possible to the theoretical setting we do not use the centering trajectory approach in our experiments.      }

\begin{figure}[h!]
	\centering
	\includegraphics[width=8cm]{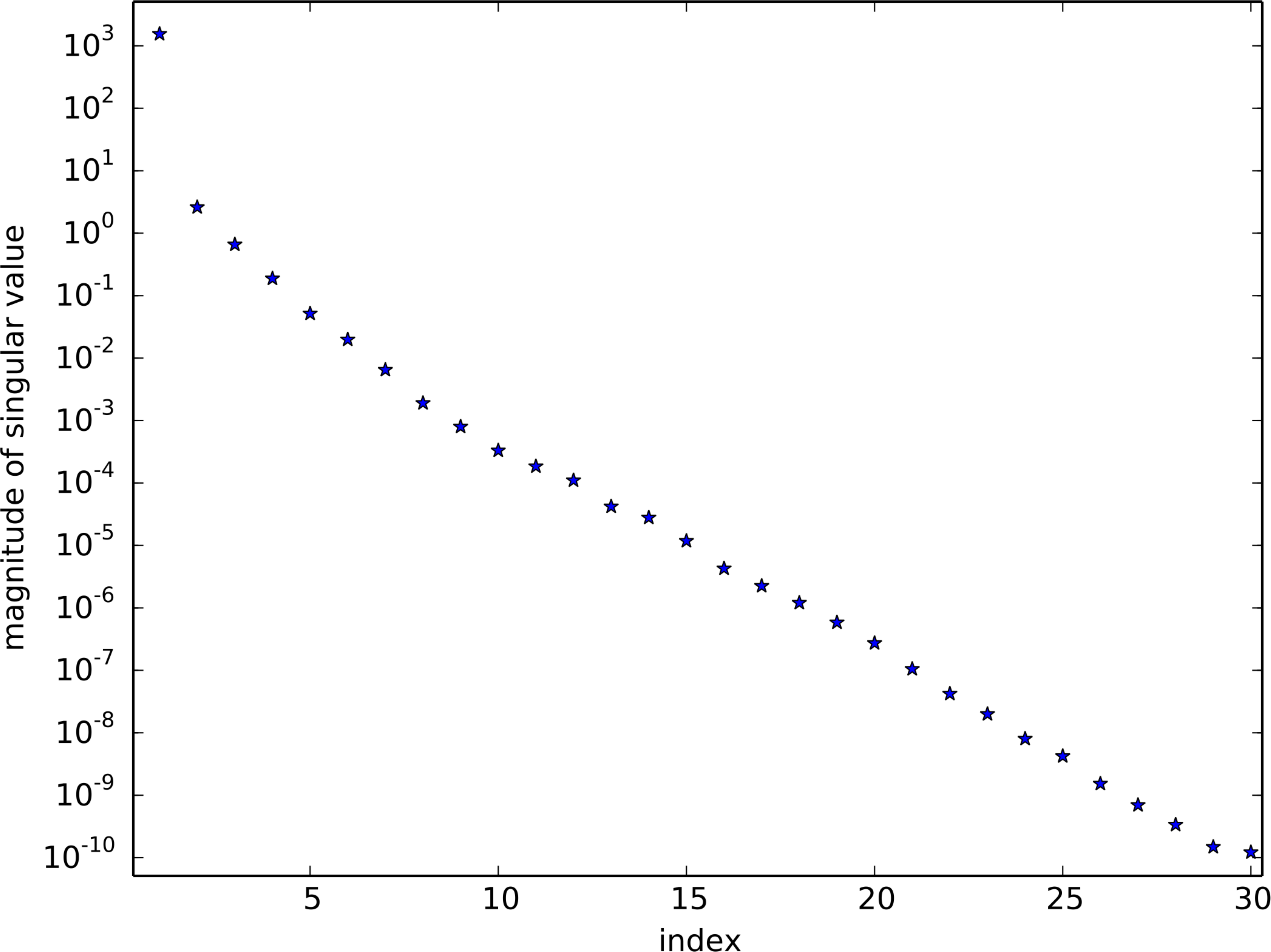} 
	\caption{The 30 largest eigenvalues for the snapshot matrix generated from the flow between two offset circles.}
	\label{eigvals_ex1_ex2}
\end{figure}

\subsection{Example 1}
\label{ex1}
We will illustrate our theoretical error estimates by performing the ``data mining'' test used in  \cite[Example 1]{GJS17}, that is, we compare the flow generated by the  {\em EnB-POD} algorithm with that generated by the {\em EnB-full} algorithm for the same perturbations, mesh, and time steps used for the generation of the POD basis. To demonstrate the accuracy of the  {\em EnB-POD} algorithm, we plot the velocity field of the ensemble average at the final time $T = 5$. In addition we plot, for both methods, the energy $\frac{1}{2}\|u\|^{2}$ and the enstrophy $  \frac{1}{2}\nu\|\nabla \times u\|^{2}$ over the time interval $0 \leq t \leq 5$.   

As can be seen in Figure \ref{ErrDiff1} (left), the difference between the average \emph{EnB-POD} and \emph{EnB-full} velocities method appears to be quite small at  $T = 5$. We can also see in Figure \ref{Energy_ex1/Enstrophy_ex1} that the energy and enstrophy of the \emph{EnB-POD} method matches that of the \emph{EnB-full} method when $6$ POD basis vectors are used in the approximation.  Lastly, it is seen in Table \ref{tabEx1}(a) that the L2 error of the approximation reduces monotonically as the number of POD basis functions used increases. 

\begin{figure}[h!]
	\centering
	\includegraphics[height=4.0cm]{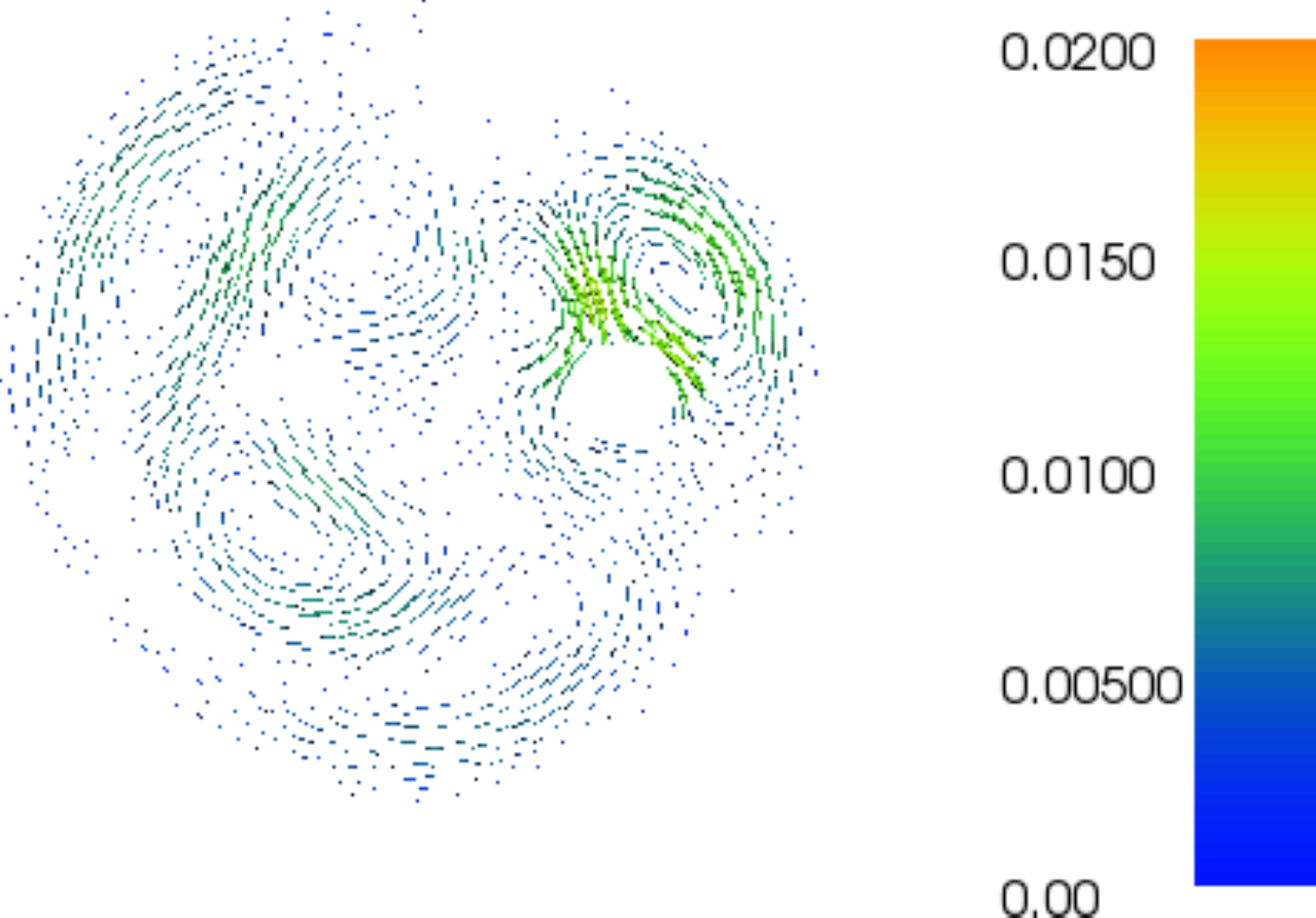}\quad
	\includegraphics[height=4.0cm]{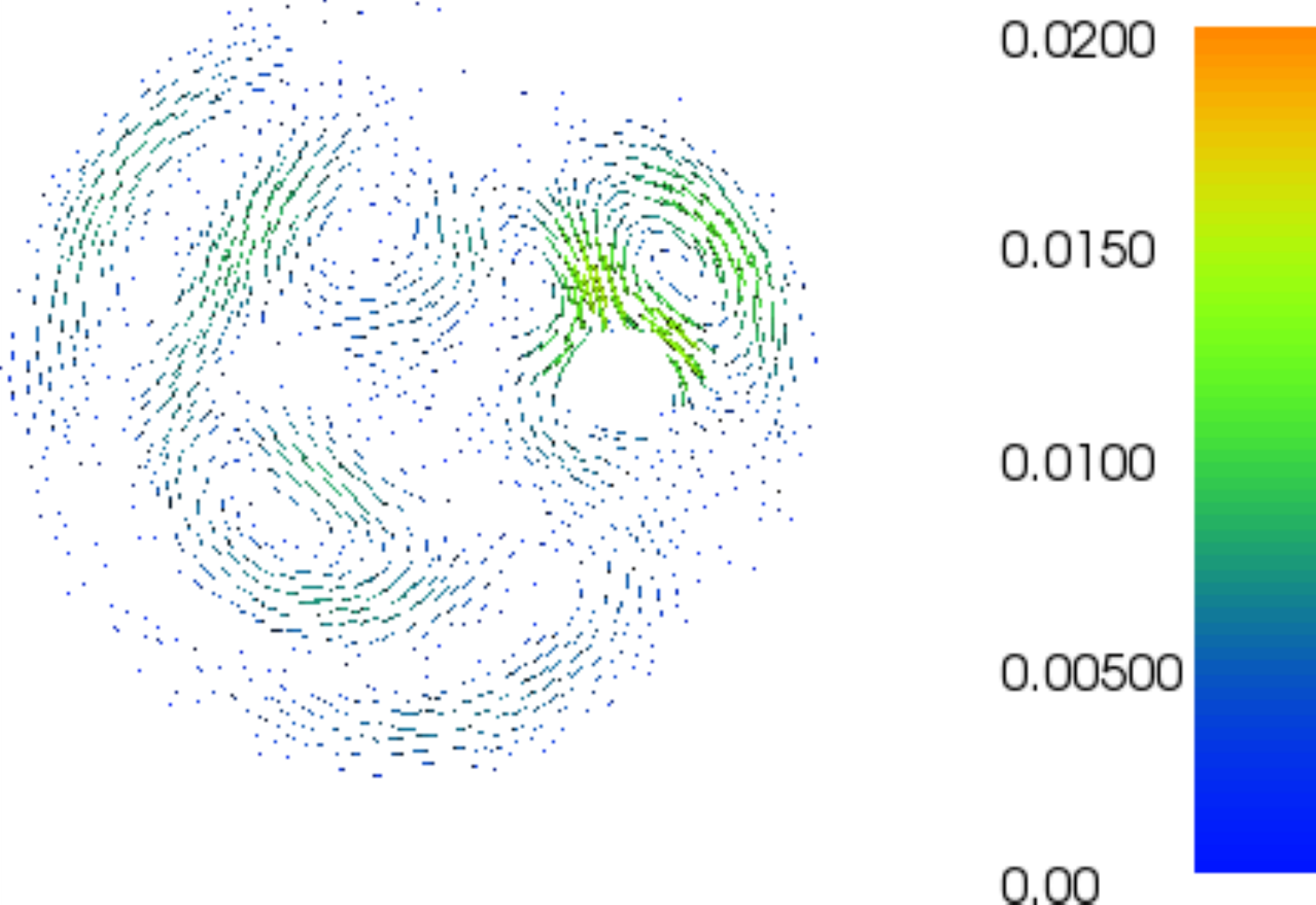} 
	\caption{The difference between velocity field at the final time $T = 5$ of the EnB-full approximation and the En-POD approximation with $6$ reduced basis vectors for example 1 (left) and example 2 (right).}
	\label{ErrDiff1}
\end{figure}

\begin{figure}[h!]
	\centering
	\includegraphics[height=4.0cm]{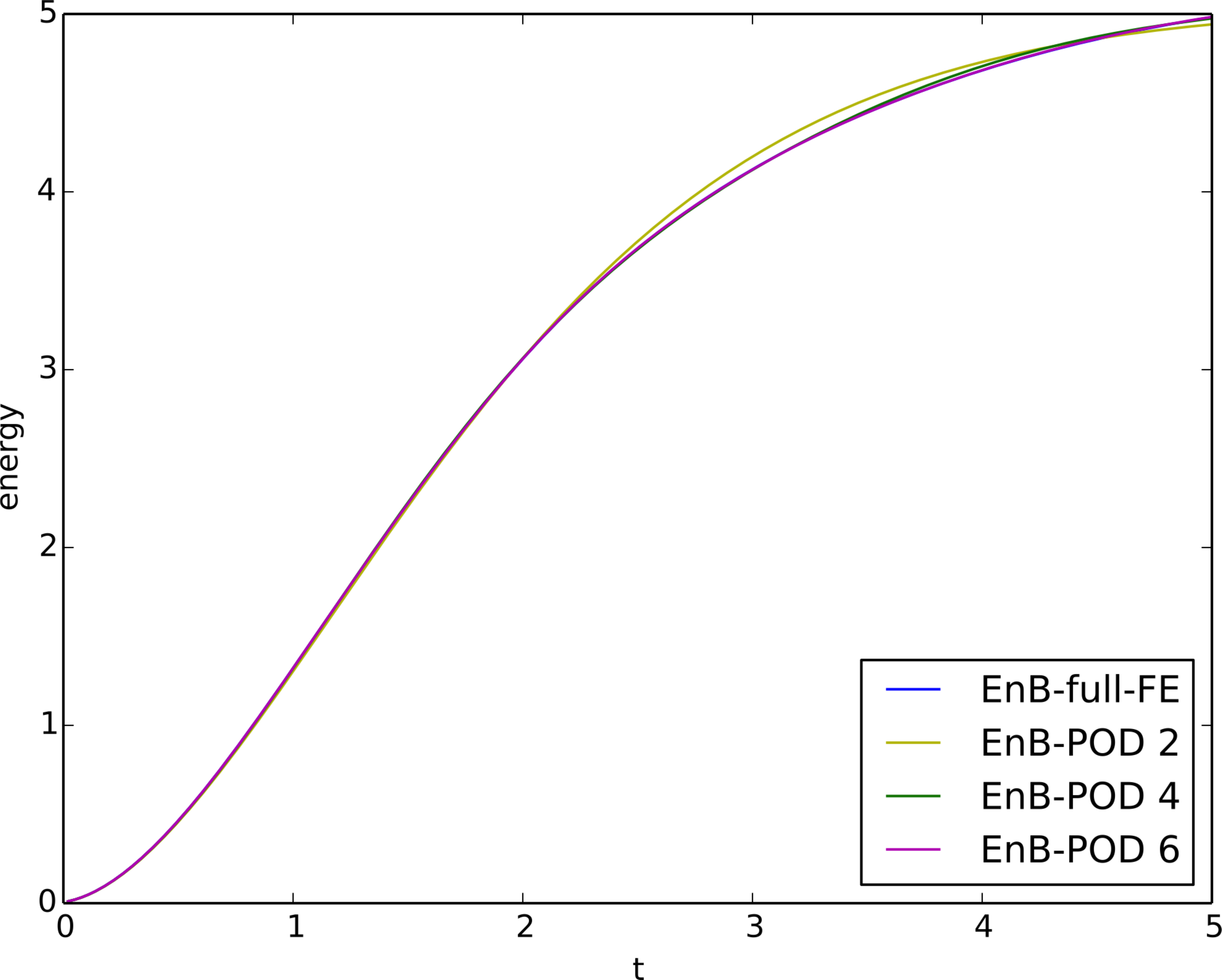}\quad
	\includegraphics[height=4.0cm]{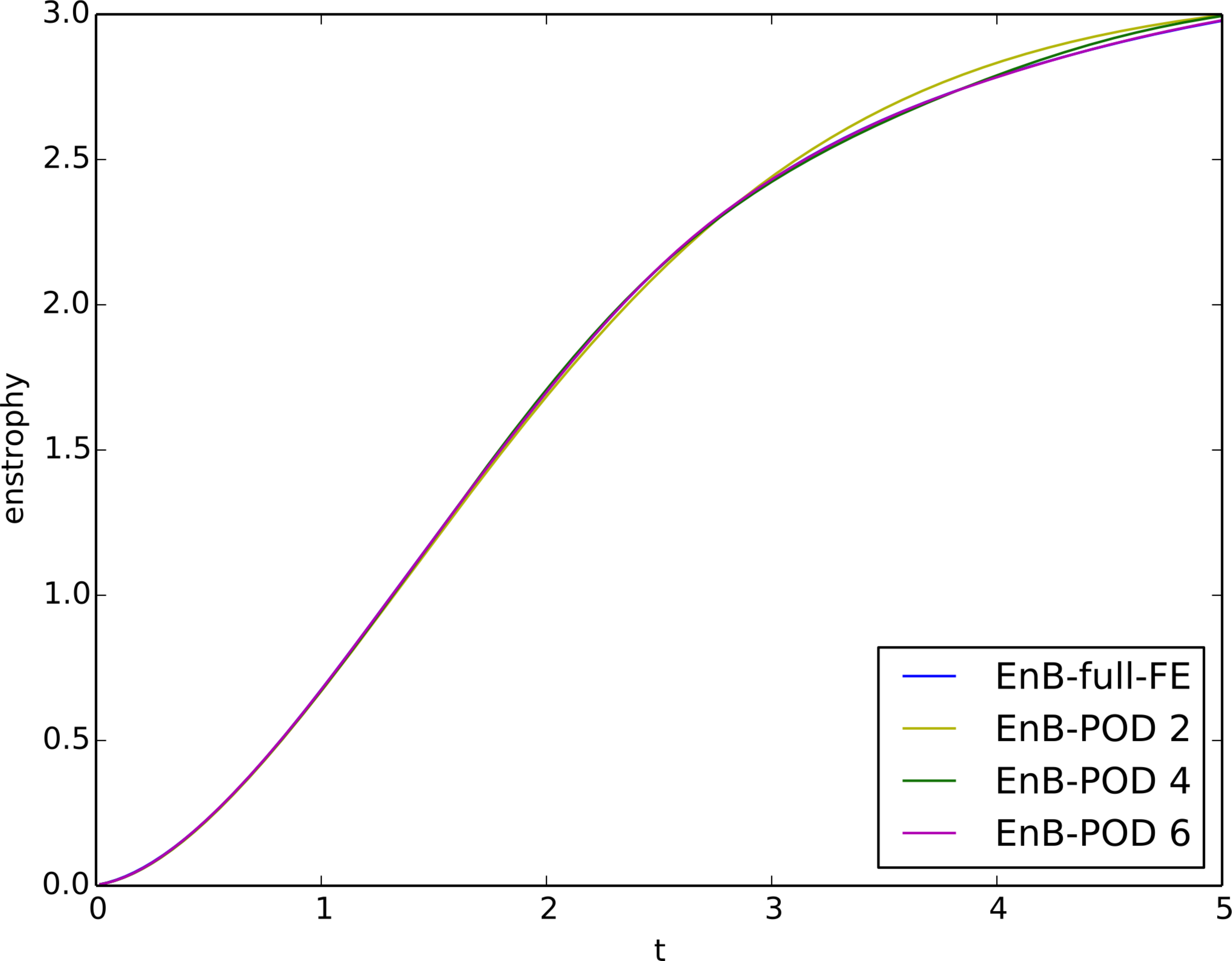}  
	\caption{For $0 \leq t \leq 5$, the energy (left) and enstrophy (right) of the ensemble determined for the EnB-full approximations and for the EnB-POD approximation of several dimensions for example 1.}
	\label{Energy_ex1/Enstrophy_ex1}
\end{figure}

\begin{table}[h!]
	\begin{center}
		\begin{tabular}{|c|c|c|c|c|c|c|c|}
			\multicolumn{2}{c}{\footnotesize(a) \em Example 1} &\multicolumn{1}{c}{}&
			\multicolumn{2}{c}{\footnotesize(b) \em Example 2}
			&\multicolumn{1}{c}{}\\
			\cline{1-2}  \cline{4-5}  
			$R$ & error && $R$ & error   \\ \cline{1-2}  \cline{4-5}   
			2 &  0.035785&& 2 & 0.035869   \\  \cline{1-2}  \cline{4-5}  
			3 &  0.021379&& 3 & 0.021437 \\  \cline{1-2}  \cline{4-5} 
			4 &  0.013802&&  4 & 0.013910    \\  \cline{1-2}  \cline{4-5}  
			5 &  0.009067&&  5 & 0.009073   \\  \cline{1-2}  \cline{4-5} 
			6 & 0.004886&&  6 & 0.004969  \\   \cline{1-2}    \cline{4-5}
		\end{tabular}
		\caption{The $L^{2}$ relative error $||u_{h}^{ave}- u_{R}^{ave}||_{2,0}$ vs. the dimension $R$ of the POD approximation for Examples 1 and 2.}\label{tabEx1}
	\end{center}
\end{table}

\subsection{Example 2}
We next wish to test our reduced basis method in the extrapolatory setting, that is, for values of the perturbation parameter $\epsilon$ different from those used to generate the reduced-order basis. To do so, we examine the problem used in Section \ref{ex1} with the reduced basis for the \emph{EnB-POD} approximation still generated by $\epsilon_{1} = 0.001$ and $\epsilon_{2} = -0.001$, but now we use five perturbations  $\epsilon_{1} =0.2$, $\epsilon_{2} = 0.4$, $\epsilon_{3} =0.6$, $\epsilon_{4} = 0.8$, and  $\epsilon_{5} = 1.0$ in the calculation of the \emph{EnB-full} and \emph{EnB-POD} solutions; all these values of $\epsilon$ are well outside the interval $[-0.001,0.001]$.

The results for the ensemble average are given in Figures \ref{ErrDiff1} (right), \ref{Energy_ex2/Enstrophy_ex2} and Table \ref{tabEx1}(b). We see that the results are identical to the first example except that, understandably, the relative error is slightly larger.

\begin{figure}[h!]
	\centering
	\includegraphics[height=4.0cm]{energy_ex2.pdf}\quad
	\includegraphics[height=4.0cm]{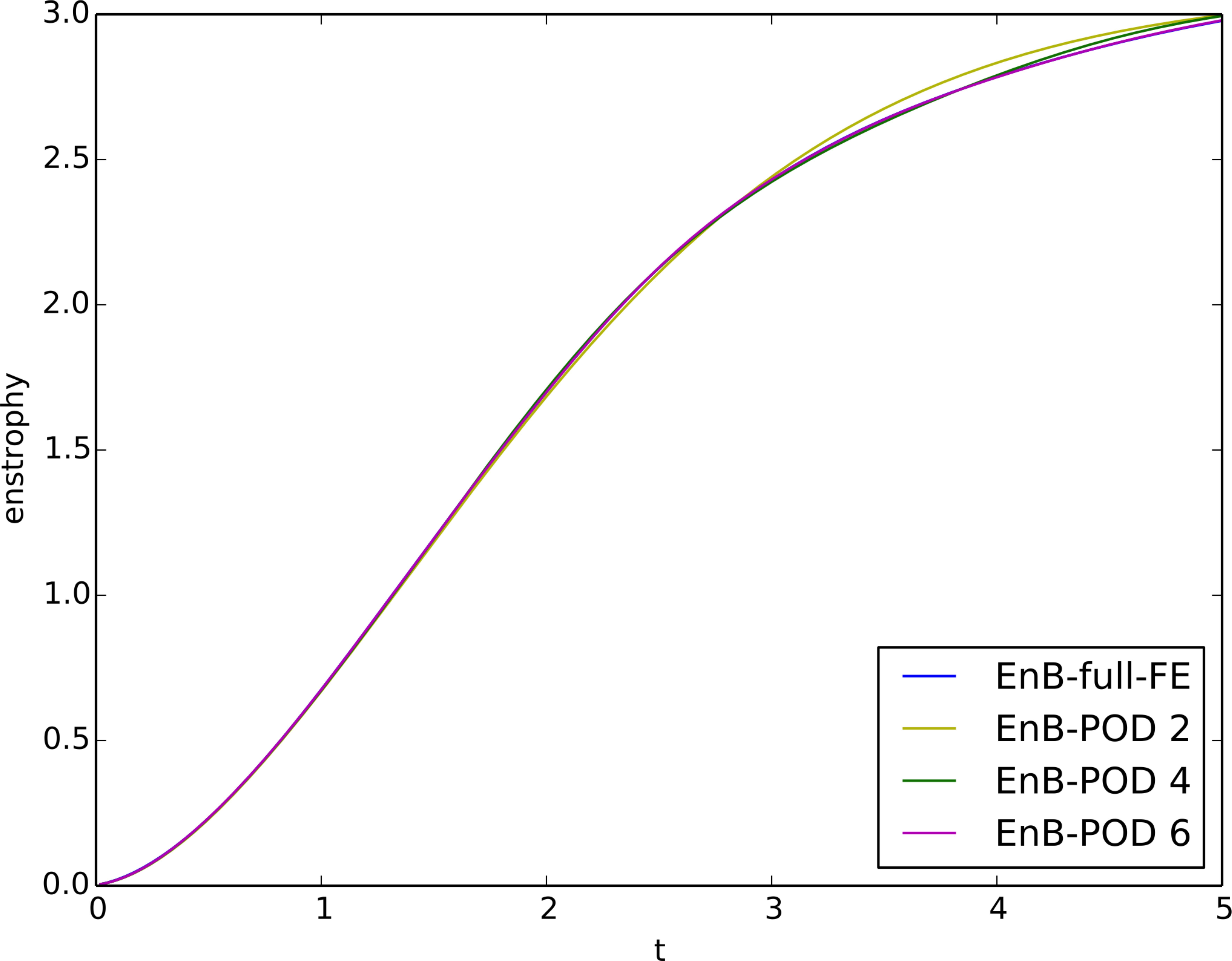}  
	\caption{For $0 \leq t \leq 5$, the energy (left) and enstrophy (right) of the ensemble determined for the EnB-full approximations and for the EnB-POD approximation of several dimensions for example 2.}
	\label{Energy_ex2/Enstrophy_ex2}
\end{figure}

\section{Conclusions}
{\color{red}In this work, a second-order ensemble-POD method for the non-stationary Navier-Stokes equations is proposed and analyzed. This is an extension of the method first proposed in \cite{GJS17} and should prove useful in settings that require long-time integrations.

The method proposed in this work as well as in \cite{GJS17} only work well with low Reynolds number flows. Whereas there is a vast array of methods for the Navier-Stokes equations for high Reynolds number flows, within the POD/ensemble framework the literature is much more limited. A few recent works within this setting are the regularized reduced-order models developed in \cite{WWXI17,XMRI17,XWWI17} and the regularized ensemble models in \cite{JL15}. A natural extension of this work will be to incorporate these regularized methods into the ensemble-POD framework.}

\section*{Acknowledgments}
The authors gratefully acknowledge the support provided by the US Air Force Office of Scientific Research grant FA9550-15-1-0001 and US Department of Energy Office of Science grants DE-SC0009324 and DE-SC0010678.

\bibliography{myrefs}
\bibliographystyle{plain}
\end{document}